\documentclass{amsart}

\usepackage{amsfonts}       
\usepackage{amsthm}
\usepackage{amssymb}
\usepackage{xcolor}
\usepackage{color}
\usepackage{multicol}
\usepackage{soul}
\usepackage[pdftex]{graphicx}
\usepackage{caption}
\usepackage{float}
\usepackage{epsfig}
\usepackage{pstricks,pst-grad}
\usepackage{graphicx}
\usepackage{todonotes}

\newtheorem{theorem}{Theorem}
\newtheorem{lemma}[theorem]{Lemma}
\newtheorem{proposition}[theorem]{Proposition}
\newtheorem{corollary}[theorem]{Corollary}
\theoremstyle{definition}
\newtheorem{definition}[theorem]{Definition}
\newtheorem{example}{Example}
\theoremstyle{remark}
\newtheorem{remark}{Remark}
\newtheorem*{theorem*}{Theorem}

\newcommand{\be}{\begin{equation}}
\newcommand{\ee}{\end{equation}}
\newcommand{\ba}{\begin{align}}
\newcommand{\ea}{\end{align}}
\newcommand{\ban}{\begin{align*}}
\newcommand{\ean}{\end{align*}}
\newcommand{\la}{\label}

\newcommand{\R}{\mathbb{R}}

\newcommand{\Hei}{\mathbb{H}}

\newcommand{\cic}{{C^{\infty}_c}}
\newcommand{\parder}[2]{\frac{\partial{#1}}{\partial{#2}}}
\newcommand{\darr}[4]{{\left\{\begin{array}{ll}
{#1}&{#2}\\
{#3}&{#4}
\end{array}\right.}}

\def\g{{g_{\lambda}}}

\def\g{{g_{\lambda}}}

\def\g1{{g_{1/2}}}

\def\ddx{{{\rm d}x}}
\def\H{{\mathbb{H}}}

\def\rd{{{\rm d}}}

\newcommand{\sph}{\mathrm{S}}
\newcommand{\ia}{({\rm i})}
\newcommand{\ib}{({\rm ii})}

\begin{document}

\thanks{M. Chatzakou is supported by the FWO Odysseus 1 grant G.0H94.18N: Analysis and Partial Differential Equations and by the Methusalem programme of the Ghent University Special Research Fund (BOF) (Grant number 01M01021), and  is a postdoctoral fellow of
the Research Foundation-Flanders (FWO) under the postdoctoral grant No 12B1223N.} 

\title {Geometric Hardy inequalities on the Heisenberg groups via convexity}

\author{G. Barbatis}
\address{Department of Mathematics, National and Kapodistrian University of Athens, \newline
15784 Athens, Greece}
\email{gbarbatis@math.uoa.gr}

\author{M. Chatzakou}
\noindent
\address{Department of Mathematics: Analysis, Logic and Discrete Mathematics, 
\newline
Ghent ~University, Belgium
}
\email{Marianna.Chatzakou@UGent.be}

\author{A. Tertikas}
\address{Department of Mathematics \& Applied Mathematics, University of Crete,
\newline 70013 Heraklion, Greece \newline
 Institute of Applied and Computational Mathematics,
FORTH, 71110 Heraklion, Greece}
\email{tertikas@uoc.gr}



%
%
%

\begin{abstract}
We prove $L^p$-Hardy inequalities with distance
to the boundary
for domains in the Heisenberg group $\H^n$, $n\geq 1$. 
Our results are based on a certain geometric 
condition. This is first
implemented for the Euclidean distance in certain
non-convex domains. It is then implemented
for the distance defined by the gauge quasi-norm
related to the fundamental solution of the horizontal
Laplacian when the domain is a half-space or a convex polytope.
Finally it is implemented for the Carnot-Carath\'eodory distance on
half-spaces and arbitrary bounded
convex domains of $\H^n$. In all cases the constant
$((p-1)/p)^p$ is obtained.
In the more general context of a stratified
Lie group
of step two we study the $p$-superharmonicity and
the weak $H$-concavity of the Euclidean distance
to the boundary, thus obtaining
an $L^p$-Hardy inequality 
on bounded convex domains which generalises previous results.
\end{abstract} 

 \maketitle

\section{Introduction} 
The classical $L^p$-Hardy inequality, $p>1$,
affirms that 
\[
\int_{\R^n} |\nabla u|^p\, {\rm d}x \geq \left|\frac{n-p}{p} \right|^p \int_{\R^n} \frac{|u|^p}{|x|^p}\,{\rm d}x\,,
\]
for $u \in C_{c}^{\infty}(\R^n\setminus\{0\})$, where the constant is sharp.

Another much studied type of Hardy inequality is where the
Hardy potential is the distance to the boundary of 
a reference domain. A well known such result
states that if
$\Omega \subset \R^n$ is a convex domain
and $d(x)={\rm dist}(x,\partial\Omega)$, then for any
$u\in\cic(\Omega)$ there holds
\[
\int_{\Omega} |\nabla u|^p\, \rd x\geq \left(\frac{p-1}{p} \right)^p \int_{\Omega}\frac{|u|^p}{d^p}\,\rd x\,,
\]
and the constant is sharp, cf. \cite{MS97}. In \cite{BFT04} the
convexity condition was replaced by the more general notion of weak mean convexity, namely the requirement
that $\Delta d \leq 0$ in the distributional sense in $\Omega$.
The above inequality is not valid without some geometric assumptions on $\Omega$ and for this
reason inequalities of this type are often called geometric Hardy inequalities.
The literature on geometric Hardy inequalities in Euclidean space is large and we 
refer the interested reader to the works 
\cite{Ba24,BEL15,Dav98,RS19} which provide an overview of the topic. 

On the other hand, subelliptic Hardy inequalities have been studied for quite a 
long time and the work \cite{GL90} of Garofalo and Lanconelli in the 90's opened 
up the research in this direction. By subelliptic Hardy inequalities, we mean 
Hardy-type inequalities considered in the setting of homogeneous Lie groups, and in 
particular stratified groups. 
The systematic analysis of homogeneous Lie groups
goes back to the seminal 
work \cite{FS82} by Folland and Stein where the authors establish the 
corresponding ``anisotropic'' non-commutative harmonic analysis, see also 
\cite{S93}. In view of their importance in the area of partial differential 
equations, stratified Lie groups, have been widely recognised as they play a key 
role in establishing subelliptic estimates for differential operators on general 
manifolds.

As in the Euclidean case, Hardy inequalities on 
stratified groups may involve
either the distance to a point or the distance to the 
boundary. Moreover, one may use the Euclidean 
distance, the Carnot-Carath\'eodory distance
or the distance related to the fundamental solution 
of the sub-Laplacian $\Delta_{H}$, often called the 
gauge pseudodistance.

Concerning the distance to a point case
we refer to \cite{CCR15,DA04b,GL90,GKY17,FP21,RS17,Y13}.
For an overview of the works in Hardy inequalities of all the above types we refer to the monograph \cite{RS19}; see also the survey article \cite{Su22}.

For Hardy inequalities involving the distance to the 
boundary the literature in the stratified setting
is limited and in most cases it involves the
Euclidean distance. In \cite{Lar16} 
 $\H^n$, S. Larson shows that
if $\Omega \subset \H^n$ is either a half space
or a convex (in the Euclidean sense) domain, then for any $u \in C_{c}^{\infty}(\Omega)$ we have 
\[
    \int_{\Omega} |\nabla_{\H^n}u|^2\,{\rm d}x \geq \frac{1}{4}\int_{\Omega} 
 \frac{|\nabla_{\H^n} d|^2}{d^2}u^2\,{\rm d}x\,,
\]
where $d(x)$, $x\in\Omega$, stands for the Euclidean 
distance to $\partial\Omega$
and the 
constant $1/4$ is the best possible. 
Later on in 
\cite{RSS20b} Ruzhansky et al. proved that if 
$\Omega$ is a half-space in $\Hei^n$ and $p>1$ then
\[
\int_{\Omega} |\nabla_{\Hei^n} u|^p {\rm d}x \geq 
    \Big( \frac{p-1}{p} \Big)^p
\int_{\Omega} \frac{|\nabla_{\Hei^n} d|^p}{d^p}|u|^p {\rm d}x  \; ,
\qquad u\in\cic(\Omega).
\]
A little earlier the same authors had showed that if $\Omega$ is
a convex domain in any stratified group then
\[
    \int_{\Omega} |\nabla_{H}u|^2\,{\rm d}x \geq \frac{1}{4}\int_{\Omega} 
 \frac{|\nabla_{H} d|^2}{d^2}u^2\,{\rm d}x\,,
\]
where $\nabla_{H}$ denotes the horizontal gradient \cite{RSS20a}.
For other results in this direction see also \cite{Rus18}; 
see also \cite{PRS19} for an alternative approach based
on a sub-Riemannian version of the Santal\'{o} formula.

Our main interest in this work is to prove subelliptic geometric Hardy inequalities with best constant
on the Heisenberg group $\H^n$, but also on
any stratified group of step two. Our approach
is based on the general method of \cite{BFT04} and
in particular in the $L^p$-superharmonicity
of the distance function. The general result is then implemented in different contexts.

Let us first consider some known results about Hardy inequalities
in Euclidean domains.
Assume that $U\subset\R^d$ is a bounded domain with $C^2$
boundary and let $p>1$.

One approach to study the Hardy inequality in $U$ is by using
the fact that if we consider functions supported in a small
neighbourhood of a boundary point, then the Hardy constant is
$((p-1)/p)^p$. However, if we try to pass from this local
inequality to a global one on $U$, then one ends up with the
inequality
\[
\int_{U} |\nabla u|^p {\rm d}x  + c\int_{U} |u|^p {\rm d}x
\geq 
    \Big( \frac{p-1}{p} \Big)^p
\int_{U} \frac{|u|^p}{d^p}{\rm d}x    \; ,
\qquad u\in\cic(U).
\]
for some $c>0$ depending on $U$. It is not possible to avoid the $L^p$
term unless some geometric assumption is made on $U$.

Indeed, for $p=2$, \cite{Dav95} Davies constructed certain planar sectors  
with Hardy constant smaller that $1/4$. By approximating such a 
sector by an increasing sequence of bounded smooth domains one concludes
the existence of bounded smooth domains with Hardy constant smaller
than $1/4$. In the same article and for dimension $d\geq 3$, Davies constructed conical domains
with Hardy constant arbitrarily small; again, this shows the
existence of bounded smooth domains with the same property.
Concerning the general case $p>1$, elemantary considarations show that when
the domain is the complement of a ball in $\R^d$ and $p>(d+1)/2$ then
the (global) Hardy constant is smaller than $((p-1)/p)^p$. A simple 
argument then
shows the existence of tori for which the same holds true.

It is a 
central question to find conditions
on the domain $U$ under which the Hardy constant is equal
to $((p-1)/p)^p$.

In the context of the Heisenberg group $\Hei^n$ one similarly expects
that in order to obtain the Hardy constant $((p-1)/p)^p$ some geometric
condition is necessary on the domain. 
Indeed, Proposition \ref{prop:non} 
below uses a lifting argument to show that each of the above examples for 
Euclidean domains induces a corresponding example in $\Hei^n$ for the 
Hardy inequality with Euclidean distance.
The main idea here is to lift the Euclidean domain $U\subset \R^n$ to the domain $U \times \R \subset \H^n$ in the Heisenberg group.

The condition introduced in this article, namely the $p$-horizontal
superharmonicity of the distance function, turns out to be a natural
condition. For the case of the Euclidean distance in a Carnot group
of step two it applies to any convex domain for any $p>1$, see
Section 4. But it is also valid to domains that are not necessarily 
convex as is seen in Section 2. Moreover it is applied, as already 
mentioned, to the Carnot-Carath\'eodory distance and the gauge
pseudodistance.

We now present the main results of this article.
The torus $T\subset\Hei^n$ defined in the next theorem is not
a convex domain but satisfies the geometric assumption
$\Delta_{p,\Hei^n}d\leq 0$ for any $p>1$.
The precise value of the constant $\beta(p,n)$
is given in Proposition \ref{prop:torus}.

\medskip

\noindent
{\bf Theorem A.} {\it Let $R>\rho>0$ and let $T$ denote the torus 
\[
T=\{ \xi= (x,y,t)\in \H^n : \; (r-R)^2 +t^2 <\rho^2 \}
\]
where $r=\sqrt{|x|^2+|y|^2}$.
For any $p>1$ there exists a positive constant
$\beta(p,n)$ such that if
\begin{align*}
\ia \qquad & 
R\geq  \rho +\Big(\frac{(2n-1)\rho}{4}\Big)^{\frac13} \\
\ib \qquad &  R\geq \beta(p,n)\rho , 
\end{align*}
then
\[
\Delta_{p,\Hei^n} \, d \leq 0\, \quad \text{in the
weak sense in}\,\, T\,,
\]
and
\[
\int_{T} |\nabla_{\H^n} u|^p {\rm d}\xi \geq 
\Big( \frac{p-1}{p} \Big)^p
\int_{T} \frac{|\nabla_{\H^n} d|^p}{d^p}|u|^p {\rm d}\xi  \; , \qquad u\in\cic(T).
\]
Moreover the constant $((p-1)/p)^p$ is sharp.}

\medskip

\black
Another class of distances in a stratified
group consists of those induced by a homogenous
quasi-norm.
In the case of the Heisenberg group $\Hei^n$
there are two particularly  important such (quasi)-norms. The  first is
\be
N(\xi) =\Big( ( |x|^2 +|y|^2)^2 +t^2 \Big)^{\frac14},
\qquad \xi=(x,y,t)\in\H^n.
\la{gauge_norm}
\ee
We note that $N^{2-Q}$, with $Q=2n+2$ being the homogeneous dimension of $\Hei^n$, is (up to a multiplicative constant) the fundamental
solution of the sub-Laplacian  on  $\H^n$ \cite{Fol73}.

A second important homogeneous quasi-norm, which respects the
sub-Riemannian geometry of $\H^n$,
is the one arising from the Carnot-Carath\'eodory distance, denoted here by $\rho$. 
Let us recall the definition of the Carnot-Carath\'eodory distance between two points $x,y \in \R^n$: For a family $\{X_1,\cdots,X_m\}$ of  vector fields, the Carnot-Carath\'eodory distance between $x,y$ is given by 
\begin{equation}\label{def:CCdista}
\rho(x,y)=\inf_{\gamma \in \mathcal{C}_{x,y}} \{ {\rm length}(\gamma)\}\,,    
\end{equation}
where $\mathcal{C}_{x,y}$ is the set of horizontal curves joining $x$ to $y$.
In the case of a stratified group when the set $\{X_1,\cdots,X_m\}$ coincides with the first stratum $V_1$ of its Lie algebra
(see Section \ref{sec:convexity} for related definitions) we always  have $d(x,y)< \infty$, and the  Carnot-Carath\'eodory distances
induced by different choices of $V_1$ are equivalent \cite[Section 1.3]{Pa89}.

Denoting by $d_N$ the distance to the boundary 
induced by $N$, cf. \eqref{ntor} below, we have the
following

\medskip

\noindent
{\bf Theorem B. }{\it (i)
Let $p>1$ and let $D \subset \H^n$ be a half-space.  There holds
\[
\int_{D} |\nabla_{\H^n} u|^p {\rm d}\xi \geq \left( \frac{p-1}{p}\right)^p
\int_{D} \frac{|\nabla_{\H^n} d_{N}|^p}{ d_N^p}|u|^p\,{\rm d}\xi \; , \qquad
u\in\cic(D)\,.
\]
(ii) In case $p=2$ the above inequality is also
valid for any bounded convex polytope. \newline
Moreover the constant is the best possible in both cases.}

\medskip
Denoting by $d_{\rho}$ the Carnot-Carath\'eodory distance (control distance) to the boundary, we obtain the following
\medskip

\noindent
{\bf Theorem C. }{\it Let $p>1$ and let $D \subset \H^n$ be a bounded and convex domain  or a half-space in $\H^n$.  There holds
\[
\int_{D} |\nabla_{\H^n} u|^p {\rm d}\xi \geq \left( \frac{p-1}{p}\right)^p
\int_{D} \frac{|u|^p}{ d_{\rho}^{p}}\,{\rm d}\xi \; , \qquad
u\in\cic(D)\,.
\]
In the case of the half-space the constant is sharp.}

\medskip

The appearance of the factor $|\nabla_{\H^n}d_N|$ in
Theorem B -- while a similar factor does not appear in Theorem C --
should be seen in the light of a result of \cite{MSC01} a
special case of which states that  $|\nabla_{\H^n}d_{\rho}|=1$
a.e. A more interesting difference between the above two homogeneous norms relates to nearest boundary points in the case
of a subspace $\Pi\subset\H^n$. The correspondence
\[
 \Pi \ni \xi \longleftrightarrow  (\xi' , s) \in \partial
 \Pi \times \R_+
\]
where $s={\rm dist}(\xi,\partial \Pi)= d(\xi,\xi')$ is
different in the two cases. In the case of the quasi-norm
\eqref{gauge_norm} the above is a simple 1-1 correspondence, as in 
the Euclidean case, cf. Proposition \ref{lem:coord}. However in the
case of the Carnot-Carath\'eodory distance the situation is
entirely different; see Proposition \ref{lem:coord2}.

In Section \ref{sec:convexity} we extend our setting 
to that of an arbitrary stratified group of step two. 
In this general setting, a certain notion of convexity
plays a central role. There are various notions of 
convexity in the sub-Riemannian setting and their
properties can vary significantly \cite{DGN03,DLZ24,LMS03}. Notably, in
\cite{MR03} R. Monty and M. Rickly prove
that in the case of the Heisenberg group $\Hei$,
if a set is geodesically convex and
contains at least three points that do not lie
on the same geodesic, then it necessarily coincides
with $\Hei$.

In our context the relevant notions
of convexity of sets and functions are the ones
introduced at the same time by Lu, Manfredi and 
Stroffolini in \cite{LMS03} on the 
Heisenberg group and by Danielli, Garofalo and Nhieu 
in \cite{DGN03} on any  stratified group. These 
notions are the analogues of the corresponding ones 
in the abelian case $\R^n$ but with a twist;
the condition refers to a convex combination of two 
elements $g,g'$ for which, 
additionally, $g' \in H_{g}$, i.e. $g'$ lies in the horizontal plane passing 
through $g$; see Section \ref{sec:Hg}
for the precise definitions.
Exploring properties of the so-called weakly $H$-concave functions
(see Definition \ref{def:H-convex.u}) and using a result
from \cite{DGN03} we prove the following theorem,
part (iii) of which generalizes \cite[Corollary 2.2]{RSS20b}.

\medskip

\noindent
{\bf Theorem D. }{\it
Let $G$ be a stratified group of step two
and let $\Omega\subset G$ be a bounded domain which is 
convex in the Euclidean sense. Then 
\begin{align*}
{\rm (i)} & \quad  \mbox{The Euclidean distance $d$ to
the boundary is weakly $H$-concave in $\Omega$} ;\\
{\rm (ii)} & \quad  \Delta_{p,H} d
\leq 0 \; \mbox{ in the distributional
sense in }\Omega ;\\
{\rm (iii)} & \quad \mbox{The Hardy inequality} \\[0.2cm]
& \hspace{1.5cm} \int_{\Omega} |\nabla_{H} u|^p {\rm d}g \geq 
\Big( \frac{p-1}{p}\Big)^p
\int_{\Omega} \frac{|\nabla_{H} 
d|^p}{d^p}|u|^p {\rm d}g  \; , 
\qquad u\in \cic(\Omega), \\[0.2cm]
& \quad \mbox{is valid.}
\end{align*}
}

\section{two general results on stratified groups}

A stratified (or Carnot) group $G\equiv \R^n$ is 
naturally a homogeneous Lie group.
Denoting by $\mathfrak{g}$ the corresponding Lie 
algebra we have $\dim(\mathfrak{g})=n$ and
$\mathfrak{g}$ admits a vector space decomposition  of the form
	 \begin{equation}
	 \label{def.carnot}
	 \mathfrak{g}=\bigoplus_{j=1}^{r} V_j\,,\quad \text{such that}\quad 
	 \left\{
	 \begin{array}{l}
	 [V_1,V_{i-1}] = V_{i}\,,
	 \quad 2\leq i\leq r,\\
	 
	 [V_1,V_r]=\{0\},
	 \end{array}
	 \right.
	 \end{equation}
where
\[
[V_i,V_j] ={\rm span}\{ [X,Y] : \; X \in V_i , \;  Y \in V_j\}.
\]
Such a stratification naturally equips $G$ with a
non-anisotropic dilation structure $\delta_\lambda:G \rightarrow G$, $\lambda>0$,
and makes $G$ a homogeneous Lie group. The vector 
spaces $V_i$ are called the strata of the Lie algebra 
$\mathfrak{g}$.
A symmetric homogeneous (quasi-)norm on $G$ is a function $N: G \rightarrow [0,
\infty)$ such that (i) $N(g)=0$ if and only if $g=e$, where $e$ is the identity 
element of $G$; (ii) $N(g)=N(g^{-1})$; and (iii)
$N(\delta_\lambda(g))=\lambda N(g)$.
In this article we shall use the term quasi-norm to indicate a symmetric 
homogeneous quasi-norm.

If $G$ is a stratified group, the system
$\{X_1,\ldots,X_m\}$, $m \leq n$, of vector fields in the first stratum  $V_1$ of $\mathfrak{g}$ generates, after iterated commutators, the whole of $\mathfrak{g}$, and so it is a system of H\"ormander vector fields on $\R^n$. The vector space spanned by $\{X_1,\ldots,X_m\}$ is  referred to as the horizontal hyperplane. 

The first-order vector-valued differential operator 
\[
\nabla_H =(X_1 ,\ldots, X_m)
\]
is then called the horizontal gradient on $G$
(or the subgradient on $G$).
Similarly ${\rm div}_H$ will denote the horizontal divergence given by
\[
{\rm div}_H (f_1,\ldots,f_m) = X_1f_1 +\ldots X_mf_m.
\]
The second-order differential operator
\[
\Delta_H =X_1^2+\ldots +X_m^2
\]
is called the horizontal Laplacian (or sublaplacian) 
on $G$ and is the sub-
Riemannian analogue of the Laplacian
on $\R^n$. By H\"ormander's 
Theorem, see \cite{Hor67}, the operator $\Delta_H$ is hypoelliptic. For $p>1$ we 
also have the associated
horizontal $p$-Laplacian given by
\[
\Delta_{p,H} \, u = {\rm div}_H  \big( | \nabla_H u|^{p-2} \nabla_H u  
\big).
\]
Finally, let us also recall that the (bi-invariant) Haar measure in the case of a stratified group is just, up to multiplication by a constant, the Lebesgue measure on the underlying manifold $\R^n$.

We first prove a general theorem which will be later applied in the
case of the Euclidean distance and of the pseudodistance induced
by the quasi-norm \eqref{gauge_norm}.

In what follows, we will say that a function is CC-Lipschitz
if it is Lipschitz with respect to the Carnot-Carath\'{e}odory
distance (equivalently, with respect to the distance
induced by any homogeneous
quasi-norm).

Part (a) of the next theorem is essentially
contained in \cite{RSS20b} but we
include the short proof of it because of its central
role in the present article.

\begin{theorem}\label{THM:Hardy,gen}
Let $G$ be a stratified group
and let $\Omega\subset G$ be open and connected. Let 
$p>1$ and let $d: \Omega \rightarrow (0,\infty)$ be a positive, 
locally CC-Lipschitz function. \newline
(a) Assume that
   \[
     \Delta_{p,H} \, d \leq 0\, \quad \text{in}\,\, \Omega\,,
   \]
where the inequality is understood in the distributional sense. Then 
     \[
          \int_{\Omega}|\nabla_{H}u|^p \ddx \geq 
   \Big( \frac{p-1}{p}\Big)^p     
     \int_{\Omega}\frac{|\nabla_H d|^p}{d^p}|u|^p\,\ddx\, ,
     \qquad u \in C_{c}^{\infty}(\Omega). 
     \]
   (b) Assume that there exist $x_0 \in \partial \Omega$ and  two
   neighbourhoods $A,A'$ of $x_0$ in $G$ with $A'\subset\subset A$ and such that 
\begin{align*}
{\rm (i)} & \mbox{ There exists $c>0$ such that  $|\nabla_H d| \geq c$
 in $A\cap \Omega$.}  \\[0.1cm]
{\rm (ii)} & \mbox{ The integral $\int_{A' \cap \Omega}d^{-1+\epsilon}$} \ddx
\mbox{ is finite for $\epsilon>0$ and diverges to $+\infty$ as $\epsilon\to 0$.}
\end{align*}
Then
\[
\inf_{u\in \cic(\Omega)} \frac{ \int_\Omega|\nabla_H u|^p \ddx}
{\int_\Omega\frac{|\nabla_H d|^p}{d^p}|u|^p\ddx}\leq 
\Big( \frac{p-1}{p}  \Big)^p .
\]
\la{thmm1}
\end{theorem}
\proof First we note that in view of \cite[Theorem 2.5]{MSC01} $\nabla_H d$ exists a.e. in $\Omega$. Let $T$ be a
vector field in $L^1_{\rm loc}(\Omega)$ and
$u\in C^{\infty}_c(\Omega)$. Using an argument from 
\cite{BFT04}, with the only difference that 
differential operators are replaced
by the corresponding horizontal ones, we obtain that
\[
 \int_{\Omega}|\nabla_{H}u|^p \ddx \geq
 \int_{\Omega} \Big( {\rm div}_H T  -(p-1)|T|^{\frac{p}{p-1}}
 \Big) |u|^p  \ddx ,
\]
where ${\rm div}_H T$ is understood in the distributional sense.
We now make the particular choice
\[
 T= -\Big( \frac{p-1}{p}\Big)^{p-1} \frac{1}{d^{p-1}}
 |\nabla_H d|^{p-2} \nabla_H d.
\]   
For this choice we have
\[
{\rm div}_H T  -(p-1)|T|^{\frac{p}{p-1}} =
\Big( \frac{p-1}{p}\Big)^{p} \frac{|\nabla_H d|^p}{d^p} 
-\Big( \frac{p-1}{p}\Big)^{p-1}\,  \frac{1}{d^p} \, \Delta_{p,H}d \,,
\]
hence (a) follows.   

To  prove (b), let $\psi$ be a smooth cut-off function supported in $A$ and
satisfying $0\leq \psi\leq 1$ and $\psi(x)=1$ in  $A'$.
We fix $\epsilon>0$, which will
eventually tend to zero, and we define
\[
u_{\epsilon}(x) =d(x)^{\frac{p-1}{p} +\epsilon} \psi(x) ,
\qquad x\in \Omega.
\]
A standard argument shows that $u_{\epsilon}$ can be used as a test
function. Applying the elementary inequality
$|a+b|^p\leq |a|^p +c_p
(|a|^{p-1}|b| +|b|^p)$, $a,b\in\R^n$,
we obtain
\begin{align*}
|\nabla u_{\epsilon} |^p &=  \Big| \big( \frac{p-1}{p}+\epsilon \big)
 d^{-\frac{1}{p}
+\epsilon}\psi \nabla_H d  + d^{\frac{p-1}{p} +\epsilon} \nabla_H\psi \Big|^p \\
&\leq \! \Big( \frac{p-1}{p} +\epsilon \!\Big)^p \! d^{-1+\epsilon p}\psi^p
|\nabla_H d|^p \! +c_p' d^{\epsilon p}|\nabla_H d|^{p-1} |\nabla_H \psi|
+c_p d^{p-1+\epsilon p}|\nabla_H \psi|^p.
\end{align*}
It follows that
\begin{align*}
\int_\Omega |\nabla_H u_{\epsilon}|^p\ddx 
& \leq \Big( \frac{p-1}{p} +\epsilon \Big)^p 
\int_\Omega d^{-1+\epsilon p}\psi^p |\nabla_H d|^p \ddx  
\\
& \quad +c_p'\int_\Omega  d^{\epsilon p}|\nabla_H d|^{p-1} |\nabla_H \psi|
\, \ddx + c_p \int_\Omega d^{p-1+\epsilon p}|\nabla_H \psi|^p \ddx.
\end{align*}
The last two integrals stay bounded as $\epsilon\to 0$, so
\[
\int_\Omega |\nabla_H u_{\epsilon}|^p\ddx  \leq
\Big( \frac{p-1}{p} +\epsilon \Big)^p 
\int_\Omega d^{-1+\epsilon p}\psi^p |\nabla_H d|^p \ddx   +O(1).
\]
We also have
\[
\int_\Omega \frac{|\nabla_H d|^p}{d^p}u_{\epsilon}^p\ddx
=\int_{\Omega} |\nabla_H d|^p d^{-1+\epsilon p}\psi^p \ddx\,.
\]
Hence, since the last integral diverges to infinity as $\epsilon\to 0$, we arrive at
\[
\frac{ \int_\Omega|\nabla_H u_{\epsilon}|^p\ddx}
{\int_\Omega \frac{|\nabla_H d|^p}{d^p}u_{\epsilon}^p\ddx}
= \Big( \frac{p-1}{p} +\epsilon \Big)^p  +o(1) , \qquad \mbox{ as }
\epsilon \to 0.
\]
Letting $\epsilon \to 0$ concludes the proof.
\endproof

Given a quasi-norm
$N$ we may define a pseudodistance by
\[
d_N(x,y) = N(y^{-1}x) \; , \qquad x,y\in G.
\]

The following proposition allows us to
apply Theorem \ref{THM:Hardy,gen} 
in the case of the gauge quasi-norm
\eqref{gauge_norm}.

\begin{proposition}
    Let $S$ be a closed set in a stratified group $G$ and let $N$ be any quasi-norm on $G$ that is smooth out of the origin. Then the pseudodistance 
\[
d_{N,S}(g) =\inf_{b\in S} d_N(b,g) = \inf_{b\in S}
N(g^{-1}b)
\]
is CC-Lipschitz.
\end{proposition}
\begin{proof}
By the equivalence of all quasi-norms on a stratified group it is enough
to show that for $g,g'\in S$ we have 
\begin{equation}
    \label{claim.Lip}
    |d(g)-d(g')|\leq K d_N(g^{-1}g')\,,
\end{equation}
for some $K>0$. By \cite[Proposition 5.14.1]{BLU07} there exists $\beta \geq 1$
such that
\[
  N(xy)\leq \beta N(x)+N(y)\,,\quad 
 \mbox{ for all}\,\, x,y \in G\, .
\]
Now, let $(b_n),(b_n)'\subset S$
be such that 
\[
d(g)= \lim N(g^{-1}b_n)\,,\qquad \quad d(g')=
\lim N(g'^{-1}b_n')\,,
\]
We then have
\begin{eqnarray*}
    d(g)-d(g')  
    & \leq  & \liminf \big[ N(g^{-1}b_n')- N(g'^{-1}b_n') \big] \\
    & = & \liminf \big[ N(g^{-1}g'g'^{-1}b_n')- N(g'^{-1}b_n') \big] \\
    & \leq & \liminf \big[ \beta N(g^{-1}g')+ N(g'^{-1}b_n')-  N(g'^{-1}b_n') \big] \\
    & = &  \beta N(g^{-1}g') \, . 
\end{eqnarray*}
Similarly we can show that 
\begin{equation*}
 d(g')-d(g)   \leq C\beta \, d_{N}(g',g)\,,
\end{equation*}
and \eqref{claim.Lip} follows.
\end{proof}

\section{Geometric Hardy inequalities on the 
Heisenberg group}
\la{Sec3}

In the section we consider geometric Hardy 
inequalities on 
the Heisenberg group $\Hei^n$. In the first part we
consider the Euclidean distance and prove the 
validity of the Hardy inequality
with best constant on certain tori under suitable assumptions on
the radii. In the second part we 
study the geometric Hardy inequality for the pseudodistance induced by the 
quasi-norm \eqref{gauge_norm}.

We recall that the  Heisenberg group $\Hei^n$
is the manifold
\[
\Hei^n =\{ \xi=(x,y,t) : x,y\in\R^n , \; t\in\R\}
\]
equipped with the group operation
\[
\xi \xi' = \big( x+ x' , y+y' , t+t' +2(x\cdot y'-y\cdot x') \big).
\]
The left-invariant vector fields
\[
X_i =\partial_{x_i} +2y_i\partial_t \; , \qquad 
Y_i =\partial_{y_i} -2x_i \partial_t\,,\quad i=1,\ldots,n,
\]
form the canonical basis
of the first stratum and the associated horizontal gradient and horizontal Laplacian
on $\H^n$ are given respectively by
\[
\nabla_{\H^n}=(X_1,\ldots,X_n,Y_1,\ldots,Y_n),
\]
and 
\[
\Delta_{\H^n} =\sum_{i=1}^n (X_i^2+Y_i^2)\,.
\]
So in the current section we denote the horizontal gradient and Laplacian by $\nabla_{\H^n}$ and $\Delta_{\H^n}$, respectively, to emphasize that the obtained results refer to the particular case of $\H^n$.

\subsection{Hardy inequalities with respect to the Euclidean distance: remarks and counterexamples}

In this section we consider geometric Hardy
inequalities in domains $\Omega\subset\Hei^n$
with respect to the Euclidean distance.

The next proposition, when combined with the examples of
Euclidean domains for which the Hardy constant is
smaller than $((p-1)/p)^p$, shows the
one cannot expect that a bounded $C^2$ domain in $\Hei^n$
has constant $((p-1)/p)^p$ unless some
geometric condition is imposed.

\begin{proposition}
\la{prop:non}
Let $U\subset\R^{2n}$ be a domain. The Hardy constant of
$\Omega =U \times \R$ as a domain in $\Hei^n$ cannot exceed
the Hardy constant of $U$.
\end{proposition}
\proof Let $v\in C^{\infty}_c(U)$ and $\psi\in\cic(\R)$ be given.
We define
\[
u(x,y,t) =v(x,y)\psi(t) \; , \qquad (x,y)\in U , \; t\in\R,
\]
A simple computation gives that for any $\epsilon>0$,
\begin{align*}
|\nabla_{\Hei^n}u|^2 & = |\nabla v|^2\psi^2 +4r^2v^2 (\psi')^2+
4 v\psi\psi' \sum_{k=1}^n (y_k v_{x_k}-x_k v_{y_k}) \\
& \leq (1+\epsilon)|\nabla v|^2\psi^2 +4(1+ \frac{1}{\epsilon})r^2v^2
(\psi')^2.
\end{align*}
Combining this with the elementary inequality
\[
|a+b|^q \leq (1+\epsilon)|a|^q +k_{\epsilon}|b|^q  \; , \quad \epsilon>0,
\]
we obtain that for any $\epsilon>0$ there exists $c_{\epsilon}>0$ such that
\[
|\nabla_{\Hei^n}u|^p \leq (1+\epsilon)|\nabla v|^p |\psi|^p
+c_{\epsilon} r^p |v|^p |\psi'|^p \; , \quad \mbox{ in }U^*.
\]
Since $d_{\Omega}(x,y,t) =d_U(x,y)$, we conclude that
\begin{align*}
&\frac{ \int_{\Omega} |\nabla_{\Hei^n}u|^p {\rm d}\xi}
{ \int_{\Omega} \frac{|u|^p}{d^p}{\rm d}\xi}
\leq (1+\epsilon) \frac{ \int_U |\nabla v|^p dx \, dy}{
{\int_U \frac{|v|^p}{d^p}dx\, dy}}
+ c_{\epsilon}\frac{ \int_U r^p|v|^p dx \, dy}{ \int_U |v|^p dx \, dy}
\cdot \frac{ \int_{-\infty}^{\infty} |\psi'|^p dt \, dy}{ 
 \int_{-\infty}^{\infty}|\psi|^p dt}.
\end{align*}
The result follows by noting that the quotient involving $\psi$ can be made arbitrarily small.
\endproof

We shall now consider the case of the torus $T$ which is not
a convex domain and has been discussed in Theorem A.
\black
For $u\in C^2(\H^n)$ we have
\be
\Delta_{\H^n}u = \sum_{i=1}^n(u_{x_i x_i} +u_{y_i y_i})
 +4(|x|^2+|y|^2)u_{tt} - 2\sum_{i=1}^n(y_i u_{x_i t} -x_i
u_{y_it})
\la{eq:0}
\ee
We now use cylindrical coordinates $(r,\omega,t)$ in $\H^n$, that is spherical coordinates
$(r,\omega)$ in $\R^{2n}$,
\[
(x,y)=r\omega  \; , \quad  r>0 , \; \omega\in \sph^{2n-1},
\]
where $\sph^{2n-1}$ denotes the unit sphere in
$\R^{2n}$.
The Euclidean gradient in $\R^{2n}$ is then given by
\[
\nabla u = u_r \, \omega +\frac{1}{r}\nabla_{\omega} u.
\]
Suppose now that a function $u\in C^2(\H^n)$ is independent of $\omega$, that is $u=u(r,t)$. In this case $\nabla u =u_r \, \omega$, so
\[
\sum_{i=1}^n(y_i u_{x_i} -x_i u_{y_i}) =
\sum_{i=1}^n(y_i \frac{u_r}{r}x_i   -x_i \frac{u_r}{r}y_i ) =0.
\]
Hence, for such functions, \eqref{eq:0} gives
\[
\Delta_{\H^n}u =u_{rr} +\frac{2n-1}{r}
u_r + 4r^2u_{tt} .
\]

Consider now a torus $T\subset\Hei^{n}$ which is symmetric with
respect to the $t$-axis and is centered at the origin. 
Letting $R,\rho$ ($R>\rho$) denote the two radii,
$T$ is  described in cylindrical coordinates as
\be
T=\{ \xi=(r,\omega,t) : \; (r-R)^2 +t^2 <\rho^2 \}.
\la{torus}
\ee
The Euclidean distance to the boundary is given by
\[
d(\xi)= \rho -\sqrt{(r-R)^2 +t^2}, \qquad \xi\in T,
\]
and is smooth in $T$ except on the
$(2n-1)$-dimensional `circle'
\[
S=\{\xi=(r,\omega,t) : \; t=0, \; r=R\}.
\]
The
horizontal Laplacian of $d$ is then given by
\be
\Delta_{\Hei^{n}}d=d_{rr} +\frac{2n-1}{r} d_r + 4r^2d_{tt},
\qquad \mbox{ in }T\setminus S.
\la{eq:1}
\ee
In $T\setminus S$ we have
\begin{align*}
d_r &= - \frac{ r-R}{\sqrt{(r-R)^2 +t^2}} \\
d_{rr} &= -\big((r-R)^2 +t^2\big)^{-\frac32} t^2 \\
d_{tt} &= -\big((r-R)^2 +t^2\big)^{-\frac32} (r-R)^2 .
\end{align*}
Substituting in \eqref{eq:1} we conclude that in $T\setminus S$ 
there holds
\begin{align}
 \Delta_{\Hei^n}d  =&  -\big((r-R)^2 +t^2\big)^{-\frac32} \frac{1}{r}
\nonumber \\
& \times \bigg\{
 \Big[ (2n-1)(r -R) +4r^3  \Big](r-R)^2 + 
 \Big[ r+(2n-1)(r -R)\Big] t^2
\bigg\}.
\la{d}
\end{align}

\begin{proposition}
\la{prop:torus}
Let $p>1$. Let $R>\rho>0$ and let $T$ be the torus
\eqref{torus}. Then there exists a positive constant
$\beta(p,n)$ such that if
\begin{align*}
\ia \qquad & 
R\geq  \rho +\Big(\frac{(2n-1)\rho}{4}\Big)^{\frac13} \\
\ib \qquad &  R\geq \beta(p,n)\rho , 
\end{align*}
then $\Delta_{p,\Hei^{n}}d\leq 0$ in $T\setminus S$. Moreover we
can take
\[
\beta(p,n)= \left\{
\begin{array}{ll}
\max\Big\{   \frac{2n+p-2}{p-1} \, , \;\frac{2n-p+1}{2(2-p)} 
 \Big\} , & \mbox{ if } 1< p<2, \\
2n \, & \mbox{ if } p=2, \\
2n+p-2  \, & \mbox{ if } p\geq\frac{ 13+ \sqrt{32n-7}}{8}.
\end{array}
\right.
\]
whereas for $2< p <\frac{ 13+ \sqrt{32n-7}}{8}$ we have
$\beta(p,n)=
1+\frac{1}{a(p,n)}$, where $a(p,n)$ is
the positive solution of
\[
(2n+p-3)^2 a^2 
  +4\Big(  (p-2)(2n+p-3) +(p-1)(2n-1) \Big)a  -4(2p-3) =0
\]
\end{proposition}
\proof Let $A =|\nabla_{\Hei^{n}}d|^2$.
We then have 
\be
\Delta_{p,\Hei^{n}} d  
= A^{\frac{p-4}{2}} \Big( A \, \Delta_{\Hei^{n}}d
+ \frac{p-2}{2} (d_r A_r +4r^2 d_t A_t) \Big).
\la{tor1}
\ee
Now, simple computations give
\be
A =d_r^2
+4r^2d_t^2 =\frac{ (r-R)^2 +4r^2t^2}{ (r-R)^2 +t^2}.
\la{a15}
\ee
and
\[
A_r = \frac{ 2(r-R)(1-4rR)t^2 +8rt^4}{ \big( (r-R)^2 +t^2 \big)^2}, \qquad 
A_t= \frac{ 2t(r-R)^2 (4r^2-1)}{ \big( (r-R)^2 +t^2 \big)^2}.
\]
Hence
\[
d_r A_r +4r^2 d_t A_t = -\Big( (r-R)^2+t^2\Big)^{-\frac52}
\bigg\{
(r-R)^2t^2(2-8r^2-8rR+32r^4) +8r(r-R)t^4
\bigg\}.
\]
Substituting in \eqref{tor1} and recalling \eqref{d} we
obtain after some more computations that
\[
\Delta_{p,\Hei^{n}} d 
= - \, A^{\frac{p-4}{2}}
\big( (r-R)^2 +t^2 \big)^{-\frac52} W \, ,
\]
where
\begin{align*}
W &= \frac{1}{r}
 \Big( (r-R)^2+4r^2t^2 \Big)
   \bigg(  (r-R)^2 \Big[ (2n-1)(r -R) +4r^3\Big]  \\
& \quad   \hspace{3.5cm} +
\Big[ (2n-1)(r -R) +r \Big] t^2  \bigg)  \\
&  \qquad  + (p-2)\bigg(
(r-R)^2 \Big[  1-4rR -4r^2 +16r^4  \Big]t^2
+4r(r-R)t^4
\bigg) .
\end{align*}
In case $p=2$ we note that our assumption 
implies that
\[
 (2n-1)(r -R) +4r^3 \geq 0 \; , \quad
 (2n-1)(r -R) +r \geq 0
\]
in $T$, hence $W\geq 0$, as required.

For $p\neq 2$ we collect similar powers of $t$ to obtain
\begin{align*}
rW &=  (r-R)^4 \Big(  (2n-1)(r-R) +4r^3 \Big)  \\
& \quad + (r-R)^2 \bigg\{ 
\Big(  2n-1 +4(2n+p-3)r^2 \Big)(r-R) \\
& \quad + 16(p-1)r^5 -8(p-2)r^3 +(p-1)r\bigg\}t^2 \\
& \quad +4r^2\Big( (2n+p-3)(r-R) +r\Big)t^4 \\
&=: C_0 +C_1 t^2 +C_2 t^4.
\end{align*}
Assumption (i) implies that $C_0\geq 0$. Similarly, assumption
(ii) implies that $R\geq (2n+p-2)\rho$ in all cases
and therefore $C_2\geq 0$.

We shall prove that $C_1\geq 0$. Equivalently, that
\[
r-R +r \, \frac{ 16(p-1)r^4 -8(p-2)r^2+p-1}{2n-1 +4(2n+p-3)r^2} \geq 0, \quad \mbox{ for all }  R-\rho \leq r \leq R+\rho.
\]
For this we shall find a positive constant $a=a(p,n)$
such that
\[
\frac{ 16(p-1)r^4 -8(p-2)r^2+p-1}{2n-1 +4(2n+p-3)r^2} \geq a \; , \quad r >0,
\]
or equivalently
\be
16(p-1)r^4 - \Big(8(p-2)+4a(2n+p-3)\Big)r^2 + p-1 -a(2n-1) \geq 0 \, ,
\;  \quad r >0.
\la{poio}
\ee
If such an $a$ has been found then we shall have $C_1\geq 0$
provided the radii of $T$ satisfy
\[
 r-R+ar\geq 0 , \quad \mbox{ for all }r\in [R-\rho,R+\rho],
\] 
which is equivalent to
\[
R \geq (1+ \frac{1}{a} )\rho.
\]
At this point we need to distinguish different cases.

\smallskip

\noindent
{\it Case $1<p<2$.} In this case we choose
\[
a =a_1(p,n):=\min\Big\{ \frac{p-1}{2n-1} , \; \frac{2(2-p)}{2n+p-3}   \Big\}
\]
which makes all coefficients in \eqref{poio} non-negative.
The requirement on the radii then is
\[
R \geq \Big(1+ \frac{1}{a_1(p,n)} \Big)\rho
=  \max\Big\{   \frac{2n+p-2}{p-1} \, , \;\frac{2n-p+1}{2(2-p)} 
 \Big\} \, \rho ,
\]
and it is satisfied by our assumptions.

\smallskip

\noindent
{\it Case $p>2$.}
In this case the coefficient of $r^2$
in \eqref{poio} is negative, so
we consider the discriminant. We have
\begin{align*}
& \Big(8(p-2)+4a(2n+p-3)\Big)^2-64(p-1)\Big(   p-1 -a(2n-1) \Big) \\
 = & 16\bigg\{
 -4(2p-3) +4\Big(  (p-2)(2n+p-3) +(p-1)(2n-1) \Big)a
 + (2n+p-3)^2 a^2 
 \bigg\}.
\end{align*}
We now choose $a=a(p,n)$ to be the positive root of the
quadratic polynomial above. So the requirement on the radii
for \eqref{poio} is
\[
R \geq \Big(1+ \frac{1}{a(p,n)} \Big)\rho 
\]
and $\beta(p,n)$ is given by
\[
\beta(p,n)=
\max\Big\{   2n+p-2 \, , \;1+ \frac{1}{a(p,n)}
 \Big\}.
\]
Finally we to note that for $p>2$,
\begin{align*}
2n+p-2 \geq  1+ \frac{1}{a(p,n)} & \; \Longleftrightarrow \;
a(p,n) \geq \frac{1}{2n+p-3} \\
& \; \Longleftrightarrow \;  4p^2-13p +11-2n\geq 0 \\
& \; \Longleftrightarrow \; p\geq \frac{13+\sqrt{32n-7}}{8}.
\end{align*}
This completes the proof.
\endproof

\begin{remark}
For $1<p<2$ we have that
\[
\frac{2n-p+1}{2(2-p)} \geq 2n+p-2  \quad \mbox{ iff }
\quad 
 \frac32 \leq p<2 .
\]
Moreover if we define
$p_0 =\sqrt{9n^2-8n+2} -3(n-1)$
then $3/2 <p_0<2$ and for $1<p<2$ we have
\[
\max\Big\{   \frac{2n+p-2}{p-1} \, , \;\frac{2n-p+1}{2(2-p)} \Big\} =
\darr{  \frac{2n+p-2}{p-1},}{ \mbox{ if }1<p\leq p_0,}
{\frac{2n-p+1}{2(2-p)},}{ \mbox{ if }p_0\leq p< 2.}
\]
\end{remark}

\begin{theorem}
\label{THM:torus}
Let $R>\rho>0$ and let $T$ denote the torus \eqref{torus}.
Let $p>1$ and assume that  conditions (i) and (ii)
of Proposition \ref{prop:torus} are satisfied.
Then
\begin{align*}
\ia & \qquad \mbox{There holds
$\Delta_{p,\Hei^{n}}d\leq 0$ in the distributional
sense in $T$.} \\
\ib & \qquad \mbox{For any $u \in C_{c}^{\infty}(T)$ there holds}  \\
&  \hspace{1.8cm}
  \int_{T} |\nabla_{\H^n} u|^p {\rm d}\xi \geq 
    \Big( \frac{p-1}{p} \Big)^p
\int_{T} \frac{|\nabla_{\H^n} d|^p}{d^p}|u|^p {\rm d}\xi .
\end{align*}
Moreover the constant in (ii) is the best 
possible. 
\end{theorem}
\proof Let $p>1$. By Proposition \ref{prop:torus}
we have $\Delta_{p,\Hei^{n}}d\leq 0$ in
$T\setminus S$. Hence the inequality
in (ii) for any $u \in C_{c}^{\infty}(T\setminus S)$
follows from  Theorem \ref{THM:Hardy,gen}.

In order to extend this to any
$u \in C_{c}^{\infty}(T)$ it is enough to establish that $\Delta_{p,\Hei^{n}}d\leq 0$ in the
distributional sense in $T$.
That is, we must prove that given a non-negative
function $\phi\in\cic(T)$ there holds
\begin{equation}
\int_T |\nabla_{\Hei^n} d|^{p-2}
\nabla_{\Hei^n} d\cdot \nabla_{\Hei^n} \phi \,
{\rm d}\xi \geq 0.
\label{111}
\end{equation}
For this we shall use a standard approximation argument.
Let
\[
q(\xi) = \sqrt{(r-R)^2 +t^2}\, , \qquad
\xi =(r,\omega,t)\in T,
\]
be the (Euclidean) distance of $\xi\in T$ to the
`circle' $S$. 
For $\epsilon>0$ small
we consider a smooth function $\psi_{\epsilon}$ on $T$
such that
\[
\psi_{\epsilon}(\xi) =\darr{0,}{\mbox{ if }q(\xi)<\epsilon,}
{1,}{\mbox{ if }q(\xi)>2\epsilon}
\]
and $|\nabla \psi_{\epsilon}|\leq c/\epsilon$.
Then $\phi_{\epsilon}:=\psi_{\epsilon}\phi$
is a non-negative smooth function
in $\cic(T\setminus S)$ and hence, by
Proposition \ref{prop:torus},
\[
\int_T |\nabla_{\Hei^n} d|^{p-2}
\nabla_{\Hei^n} d\cdot \nabla_{\Hei^n}
\phi_{\epsilon} \, {\rm d}\xi \geq 0.
\]
Since $|\nabla_{\Hei^n} d|$ is bounded,
in order to complete the proof it is enough to show that
\[
 \int_T  | \nabla_{\Hei^n}\phi_{\epsilon} -\nabla_{\Hei^n} \phi| \, {\rm d}\xi \longrightarrow 0 \; , \quad \mbox{ as }\epsilon\to 0,
\]
since \eqref{111} will then follow by
letting $\epsilon\to 0$.

In fact, since $|\nabla_{\Hei^n}u |\leq c |\nabla u |$ in $T$, it is enough to consider the
Euclidean gradient. We have
\[
\| \nabla \phi_{\epsilon}-\nabla\phi\|_{L^1(T)}
\leq \| (1-\psi_{\epsilon})\phi \|_{L^1(T)} +  
\|\phi \nabla \psi_{\epsilon}\|_{L^1(T)}.
\]
The first norm in the RHS tends to zero as
$\epsilon \to 0$ by the Dominated Convergence Theorem.
For the second one we have
\begin{align*}
\int_T | \phi\nabla \psi_{\epsilon}| {\rm d}\xi 
&\leq \frac{c}{\epsilon} \int_{ \{ \epsilon
<q(\xi)<2\epsilon\} }  {\rm d}\xi \\
& \leq \frac{c_1}{\epsilon} \epsilon^2 \\
& \to 0.
\end{align*}
Hence the Hardy inequality (ii) has been proved.

To establish the optimality of the constant we
apply part (b) of Theorem  \ref{THM:Hardy,gen}.
Assumption (i) is satisfied by
\eqref{a15}.
The fact that (ii) is satisfied is well known, see \cite[Lemma 5.2]{BFT04}.
This completes
the proof of the theorem.
\endproof

\begin{remark}
It is evident from the argument in the above proof that
if $\Omega\subset\Hei^n$ is a domain with $C^2$ boundary
then the corresponding Hardy constant cannot be larger
than $((p-1)/p)^p$,
provided there exists a point $\xi_0\in
\partial\Omega $ such that $\nabla_{\Hei^n}d(\xi_0)\neq 0$.
This is very generic. We refer to \cite{M06} and references
therein for results concerning the negligibility of the set of
such points in a much wider context.

In the case of the Heisenberg group $\Hei^n$ we can give
a direct proof of this negligibility.
Indeed, if $\xi_0\in\partial\Omega$ is a boundary
point with $\nabla_{\Hei^n}d(\xi_0)=0$, then
we have
\[
d_{x_i}^2 =4y_i^2d_t^2 \; , \qquad
d_{y_i}^2 =4x_i^2d_t^2 \; , \quad i=1,\ldots, n,
\quad \mbox{ at the point }\xi_0.
\]
Since $|\nabla d|=1$, adding implies
\[
(1+4r^2)d_t^2=1 ,
\quad \mbox{ at the point }\xi_0.
\]
It follows that if in addition the
tangent hyperplane at $\xi_0$ is parallel
to the hyperplane $\{t=0\}$
(and so $d_t(\xi_0)=1$), then we must
necessarily have $r=0$, that is the point
$\xi_0$ must lie on the $t$-axis.

\end{remark}

\subsection{Hardy inequalities with respect to the gauge pseudodistance}
\label{sec:gauge}

In this section we consider geometric Hardy 
inequalities in the Heiseberg group with respect to 
the gauge  quasi-norm 
\[
N(\xi) =\Big( (|x|^2+|y|^2)^2 +t^2 \Big)^{\frac14} , \qquad \xi =(x,y,t) \in \H^n\,.
\]
For a given domain $\Omega\subset\Hei^n$ the induced 
distance to the boundary is given by
\be
d_{N}(\xi) = {\rm dist}_{N}(\xi,  \partial\Omega)
=\inf\{ N((\xi')^{-1}\xi) , \; \xi' \in \partial\Omega\},
\qquad \xi\in\Omega.
\la{ntor}
\ee

We first consider the case where our domain
is the half-space
\[
\Pi_0 =\{ (x,y,t) \in \Hei^n : t>0\}.
\]
It is then easy to see that for $\xi=(x,y,t)\in \Pi_0$
and $\xi'=(x',y',0)\in\partial \Pi_0$ we have
\be
d_{N}(\xi, \xi')= \bigg(
\Big( |x'-x|^2+|y'-y|^2 \Big)^2 +\Big(t+ 2(x\cdot y'-y \cdot x')\Big)^2
\bigg)^{\frac14}.
\la{sp1}
\ee

\begin{lemma}\label{lem:d_Nformula}
Let $\Pi_0 =\{ (x,y,t) \in \H^n : t>0\}$ and let 
\be
d_N(\xi) = \inf\{ N((\xi')^{-1}\xi) , \;  \xi' \in 
\partial\Pi_0\} \; , \qquad \xi\in\Pi_0,
\la{inf}
\ee
denote the corresponding pseudodistance to the boundary.
Then, for any $\xi=(x,y,t)\in \Pi_0$,
$d_N(\xi)$ depends only on $r=\sqrt{|x|^2+|y|^2}$ and $t>0$. More precisely, we have
\be 
\la{soug}
d_N(\xi)=d_N(r,t) = 
\darr{\big( 2r^4s^2 -3tr^2s+t^2 \big)^{\frac14} ,}{ r>0, \; t>0,}
{t^{\frac12},}{ r=0 , \; t>0,}
\ee
where for fixed $r,t>0$ the real number $s \in \R$ is the unique 
solution of the equation
\be
s^3 +
2s    - \frac{t}{r^2}=0 \; . 
\la{sp2}
\ee
\la{lem:ex}
\end{lemma}
\proof
The case $r=0$ is immediate, so we assume that $r>0$.
By \eqref{sp1} the infimum in \eqref{inf} is attained at a point
$(x',y')$ which is a critical point of the function
\[
F(x',y') =\Big( |x'-x|^2+|y'-y|^2 \Big)^2 +\Big(t+ 2(x\cdot y'-y \cdot x')\Big)^2 
, \quad (x',y')\in\R^{2n}.
\]
For $i=1,\ldots,n$ we have
\begin{align*}
F_{x_i'} &=  4 \Big( |x'-x|^2+|y'-y|^2 \Big)(x_i'-x_i)-
4\Big(t+ 2(x\cdot y'-y \cdot x')\Big)y_i \, ,\\
F_{y_i'} &=  4 \Big( |x'-x|^2+|y'-y|^2 \Big)(y_i'-y_i) +
4\Big(t+ 2(x\cdot y'-y \cdot x')\Big)x_i \, .
\end{align*}
Assume now that $(x',y')$ is a critical point of $F$.
Then necessarily $(x',y')\neq (x,y)$.
From the last two relations we then obtain
\[
x_i(x_i'-x_i) +y_i(y_i'-y_i)=0 , \qquad i=1,\ldots , n.
\]
We set
\[
s= \frac{ t+ 2(x\cdot y'-y \cdot x')}{ |x'-x|^2+|y'-y|^2 }.
\]
and note that
\be
 x_i'-x_i = sy_i \; , \qquad y_i'-y_i =-sx_i \; , \qquad  i=1,\ldots , n.
\la{uni}
\ee
We then have
\be
x\cdot y'-y \cdot x'  =\sum_{j=1}^n 
\big[ x_j( y_j'-y_j) -y_j(x_j'-x_j) \big]
= -r^2s 
\la{da6}
\ee
and 
\[
|x'-x|^2+|y'-y|^2  = s^2\sum_{j=1}^n (y_j^2 +x_j^2)= r^2s^2 .
\]
Hence for $i=1,\ldots,n$, we have
\[
F_{x_i'} =  4y_i (r^2s^3 +2r^2s-t) , \qquad
F_{y_i'} =  -4x_i (r^2s^3 +2r^2s-t)
\]
and we thus conclude that $s$ must solve \eqref{sp2}.

Since the cubic equation has a unique solution, there exists a
unique critical point $(x',y')$ of $F$  given by \eqref{uni}.

Finally, by  \eqref{da6} we have
\begin{align}
d_{N}^4(\xi) & = \Big(|x'-x|^2+|y'-y|^2 \Big)^2 +\Big( t+
2(x\cdot y'-x'\cdot y) \Big)^2 \nonumber \\
& = \Big(|x'-x|^2+|y'-y|^2 \Big)^2 +\Big( t-2r^2s \Big)^2 \nonumber \\
&= 2r^4s^2 -3tr^2s+t^2, \la{sp4}
\end{align}
where we have also used \eqref{sp2}.
This completes the proof.

\endproof

\begin{proposition}\label{PROP:half-space}
Let $\Pi \subset\Hei^n$ be an arbitrary half-space
and let $d_{N}(\xi)$, $\xi\in\Pi$,
denote the pseudodistance to the boundary
with respect to the quasi-norm $N$. Then for any $p>1$ there holds
$\Delta_{p, \H^n} d_{N}\leq 0 $ in the distributional sense in $\Pi$.
\la{thmm2}
\end{proposition}
\proof By group action
(see also \cite[p340]{Lar16}) 
it is enough to consider the case $\Pi=\Pi_0$.
Also, it is preferable to work with the function
\[
G(r,t)=d_N(r,t)^4 
\]
instead of $d_N(r,t)$. 
To compute the various derivatives of $G(r,t)$
we recall from Lemma \ref{lem:ex} that
\be
G(r,t) =2r^4s^2 -3tr^2s+t^2 =(t-r^2s)(t-2r^2s) 
\quad \mbox{ in }\Pi_0 \setminus \{ r=0\},
\la{aek}
\ee
where $s=s(r,t)$ is defined by  \eqref{sp2}.
Since
$ t=r^2(s^3+2s)$ we may
eliminate $t$ from \eqref{aek} and we obtain
\be
\la{aek10}
G(r,t)=r^4s^4(s^2+1).
\ee
By \eqref{sp2} we have
\be
s_t =\frac{1}{r^2(3s^2+2)} \; , \qquad
s_r = - \frac{2s(s^2+2)}{r(3s^2+2)}.
\la{sp3}
\ee
Hence
\begin{align}
G_r & = 4r^3s^4(s^2+1) +r^4(6s^5+4s^3)s_r \nonumber\\
&= 4r^3s^4(s^2+1) -r^4(6s^5+4s^3) \frac{2s(s^2+2)}{r(3s^2+2)}
\nonumber  \\
&= -4r^3s^4. \la{aek4}
\end{align}
Similarly we obtain
\begin{align}
G_t &= 2r^2s^3  , \qquad  \qquad
G_{rr} =   - \frac{4r^2s^4(s^2-10)}{3s^2+2}  \la{aek5} \\
G_{tt} &=   \frac{6s^2}{3s^2+2} , \qquad \quad
G_{rt} =  -\frac{16rs^3}{3s^2+2}. \nonumber
\end{align}

Setting $A=|\nabla_{\Hei^{n}}d_N|^2$ we have, cf. \eqref{tor1},
\be
\Delta_{p,\Hei^{n}} d_N  =  A^{\frac{p-4}{2}}
 \Big( A \, \Delta_{\Hei^{n}}d_N
+ \frac{p-2}{2} \big(d_{N,r} A_r +4r^2 d_{N,t} A_t\big) \Big).
\la{aek1}
\ee
We have
\[
d_{N,r} = \frac14 G^{-\frac34} G_r   , \qquad  d_{N,t} = 
\frac14 G^{-\frac34} G_t , \qquad
d_{N,rr} = -\frac{3}{16}G^{-\frac74}G_r^2  +\frac14 G^{-\frac34}G_{rr} 
\]
and
\[
d_{N,tt} = -\frac{3}{16}G^{-\frac74}G_t^2  +\frac14 G^{-\frac34}
G_{tt} , \qquad 
d_{N,rt} = -\frac{3}{16}G^{-\frac74}G_r G_t  +\frac14 G^{-\frac34}
G_{rt} .
\]
Moreover
\[
A =d_{N,r}^2+4r^2d_{N,t}^2 
 = \frac{1}{16}G^{-\frac32}(G_r^2+4r^2G_t^2)
=G^{-\frac32} r^6s^6(s^2+1),
\]
\begin{align*}
A_r & = 2d_{N,r}d_{N,rr} +8rd_t^2 +8r^2d_{N,t}d_{N,rt} \\
& =G^{-\frac52} \Big(
 -\frac{3}{32}G_r^3 + \frac18 G \, G_r G_{rr} +
 \frac{r}{2} G\, G_t^2 -\frac{3r^2}{8}G_r G_t^2 + \frac{r^2}{2}G \, G_t G_{rt} \Big) \\
& =G^{-\frac52} \, \frac{2r^9s^{12}(s^2+2)(s^2+1)}{3s^2+2} 
\end{align*}
and
\begin{align*}
A_t & = 2d_{N,r}d_{N,rt}  +8r^2d_{N,t}d_{N,tt} \\
& =G^{-\frac52} \Big(
 -\frac{3}{32} G_r^2 G_t + \frac18 G\, G_r G_{rt}
  -\frac{3r^2}{8}G_t^3 +\frac{r^2}{2}G \, G_t G_{tt} \Big)\\
&= - G^{-\frac52} \, \frac{r^8s^{11}(s^2+1)}{3s^2+2}.
\end{align*}
Combining the above we arrive at
\be
d_{N,r} A_r +4r^2 d_{N,t} A_t =- 2\, G^{-\frac{13}{4}} \,
\frac{r^{12} s^{14}(s^2+1)^3}{3s^2+2}.
\la{aek2}
\ee
On the other hand
in $\Pi_0 \setminus  \{ r=0\}$
we have, cf. \eqref{eq:1},
\begin{align}
\Delta_{\H^n} d_N & = d_{N,rr} +\frac{2n-1}{r} d_{N,r} + 4r^2d_{N,tt} 
\nonumber \\
&= \frac{1}{4}G^{-\frac{7}{4}}
\Big\{
GG_{rr} -\frac34 G_r^2  +\frac{2n-1}{r} GG_r +4r^2 GG_{tt} -3r^2 G_t^2
\Big\} \nonumber \\
&= - G^{-\frac74} \,  \frac{r^6 s^8(s^2+1)\big( (6n-2)s^2+4n-3 \big)}{3s^2+2}. 
\la{aek3}
\end{align}
From \eqref{aek1}, \eqref{aek2} and \eqref{aek3} we obtain
\[
\Delta_{p,\Hei^{n}} d_N  =   -  G^{-\frac{3p+1}{4}}
\,  \frac{r^{3p} s^{3p+2}(s^2+1)^{\frac{p}{2}}\big( (6n+p-4)s^2+4n+p-5 \big)}{3s^2+2} \leq 0,
\]
and the desired inequality has been proved pointwise in $\Pi_0
\setminus\{r=0\}$ (where $d_N$ is smooth).
To complete the proof we argue as in the proof of Theorem
\ref{THM:torus}, using in particular functions $\psi_{\epsilon}$, $\epsilon>0$,
as in that proof. Part (i) is also used at this point since the local boundedness of $|\nabla_H d_N|$ is required when letting $\epsilon\to 0$.
 \endproof

\begin{proposition}
Let $\Pi_0=\{ (x,y,t) \in \H^n : t>0\}$ and let
$d_N=d_N(r,t)$
denote the corresponding gauge pseudodistance
to the boundary $\partial \Pi_0$ of the point
$\xi=(r,\omega,t)\in\Pi_0$ expressed in cylindrical coordinates.
Then for any fixed $r\neq 0$ we have
\begin{align*}
{\rm (i)} & \quad d_{N}(r,t) =\frac{t}{2r} +O(t^3) \\
{\rm (ii)} & \quad |\nabla_{\H^n} d_{N}(r,t)| =1 +O(t^2)
\end{align*}
as $t\to 0+$.
\la{propp3}
\end{proposition}
\proof
Differentiating \eqref{sp4} we get
\[
4d_N^3d_{N,t} =2t -3r^2s +(4r^4s-3tr^2)s_t. 
\]
Now using the first part of \eqref{sp3} and the fact that $s$ solves \eqref{sp2} we obtain
\be
\la{soug1}
d_{N,t} =\frac{6ts^2 -2r^2s +t-9r^2s^3}{4d^3(3s^2+2)} = 
\frac{6ts^2 +16r^2s -8t}{4d^3(3s^2+2)}.
\ee
Similarly we find that
\be
\la{soug2}
d_{N,r} =\frac{24r^4s^4 +16r^4s^2 -18tr^2s^3 -20tr^2s
 +6t^2}{4d^3r(3s^2+2)} = \frac{ 40r^2ts -32r^4s^2-12t^2  }{4d^3r(3s^2+2)}.
\ee
We now let $t\to 0+$. From \eqref{sp2} we find
\[
s= \frac{t}{2r^2} -\frac{t^3}{16r^6} +O(t^5).
\]
Plugging this in \eqref{sp4} we have
\[
d_N^4(r,t) =\frac{t^4}{16r^4} -\frac{t^6}{64r^8} +O(t^8)
\]
and (i) follows. We then also have
\[
\frac{1}{d_N^6(r,t)}  =\frac{64r^6}{t^6} +O(\frac{1}{t^4})
\]
and combining the above we obtain
\[
d_{N,t}^2 =\frac{1}{4r^2} +O(t^2) \; , \qquad d_{N,r}^2 =\frac{t^2}{4r^4}
+O(t^4).
\]
We thus conclude that
\[
|\nabla_{\H^n} d_N|^2 =d_{N,r}^2+4r^2 d_{N,t}^2 =1 +O(t^2),
\]
as required.
\endproof

\begin{remark}
\la{kos}
We note for later use that we can use \eqref{sp2} to eliminate $t$ from
the expressions \eqref{soug1}, \eqref{soug1} and \eqref{soug2} for
$d_N$, $d_{N,t}$ and $d_{N,r}$ respectively. We obtain
\[
d_N = rs(s^2+1)^{1/4}, \qquad
d_{N,r}= \frac{r^2s^3}{2d^3}, \qquad
d_{N,t} = \frac{r^2s^3}{2d^3}.
\]
\end{remark}

\begin{theorem}\label{thm:half-space.d_N}
Let $p>1$ and $\Pi$ be an arbitrary half-space in $\Hei^n$. Let
$d_{N}(\xi)={\rm dist}_{N}(\xi,\partial \Pi)$
denote the corresponding
pseudodistance of $\xi \in \Pi$ to the boundary $\partial \Pi$.
Then there holds
\[
\int_{\Pi} |\nabla_{\H^n} u|^p {\rm d}\xi \geq 
\Big( \frac{p-1}{p}  \Big)^p
\int_{\Pi} \frac{|\nabla_{\H^n} d_{N}|^p}{d_{N}^p}|u|^p {\rm d}\xi \; , \quad
u\in\cic(\Pi).
\]
Moreover the constant is the best possible.
\end{theorem}
\proof
Action by an appropriate group element reduces the proof
to the case
$\Pi =\Pi_0=\{ (x,y,t) : t>0\}$. The validity of
the Hardy inequality is a consequence of
Theorem
\ref{thmm1} (a) and Proposition \ref{thmm2}.
The sharpness of the constant follows from the second part of
Theorem  \ref{thmm1} (b) and Proposition \ref{propp3}.
\endproof

In case $p=2$ we will extend the above to the
case of a bounded convex polytope.
For this we will need the following lemma where,
as above, $\Pi_0=\{(x,y,t)\in\Hei^n : \; t>0\}$.
\begin{proposition}
    Any point $\xi\in \Pi_0$ has a unique nearest boundary
point $(x',y',0)\in \partial\Pi$. Moreover, given a point
$\xi'=(x',y',0)\in \partial\Pi_0$ and $\rho>0$, there exists a unique point
$\xi=(x,y,t)\in\Pi_0$ whose nearest boundary point is $\xi'$ and for which $d_N(\xi)=\rho$.
\la{lem:coord}
\end{proposition}

\proof We have already seen in the proof
of Lemma \ref{lem:ex}
that given $\xi\in\Pi_0$ the nearest boundary
point $\xi'\in\partial\Pi$ is uniquely defined.

Suppose now that a point $\xi'=(x',y',0)
\in\partial\Pi_0$ and $\rho>0$ are given.
Assume that $\xi\in \Pi_0$ has $\xi'$
as its nearest boundary point and that $d_N(\xi)=\rho$.
Denoting $r^2=|x|^2 + |y|^2$ and $r'^2=|x'|^2 + |y'|^2$
we  have from \eqref{uni} that
\[
r'^2 =(1+s^2)r^2 \,  ,
\]
where $s>0$ is defined in terms of $r,t$ by
\eqref{sp2}.
We also have (cf. \eqref{aek10}) $\rho^4=r^4s^4(1+s^2)$.
We thus conclude that
\[
\frac{r'^4}{\rho^4} = \frac{1+s^2}{s^4},
\]
and this relation uniquely determines $s>0$. Now going back to \eqref{uni}
we  obtain
\[
x_i =\frac{x_i' -sy_i'}{1+s^2} , \qquad y_i= \frac{sx_i' +y_i'}{1+s^2}.
\]
We also have $t=r^2(s^3+2s)$, hence the point $\xi =(x,y,t)
\in\Pi_0$ has been uniquely determined.
It is not difficult now to see that 
this point
has indeed $(x',y',0)$ as its nearest boundary point
and $d_N(\xi)=\rho$. This completes the proof.
\endproof
\begin{remark}
    Let us point out that the convexity
of a set $\Omega \subset G$ in a stratified group $G$ is a genuine geometric notion in the sense that
it is invariant under left translations; i.e., if $\Omega\subset G$ is convex and $g\in G$, then $g\Omega$ is also convex. 
\end{remark}

\begin{theorem}\label{THM:Hardy,N,Omega}
Let $\Omega\subset\Hei^n$ be a bounded convex polytope 
and let $d_{N}(\xi)$, $\xi \in \Omega$,
denote the corresponding gauge pseudodistance to the 
boundary. Then 
\begin{align*}
{\rm (i)} & \quad  \Delta_{\H^n} d_{N} \leq 0 \; \mbox{ in the distributional
sense in }\Omega .\\
{\rm (ii)} & \quad \mbox{The Hardy inequality}
\\[0.2cm]
& \hspace{1.5cm}
 \int_{\Omega} |\nabla_{\H^n} u|^2 {\rm d}\xi \geq 
 \frac14
\int_{\Omega} \frac{|\nabla_{\H^n} d_N|^2}{d_N^2}
u^2
 {\rm d}\xi  \; ,  \quad u\in \cic(\Omega), \\[0.2cm]
& \quad \mbox{is valid, and the constant $1/4$ is sharp.}
\end{align*}
\end{theorem}

\proof Let
$E_1,\ldots,E_m$ denote the sides of $\Omega$.
We define
\[
A_k =\{ \xi\in\Omega : d(\xi) ={\rm dist}(\xi, E_k) ,
\}, \qquad   k=1,\ldots ,m.
\]
Hence the sets $A_k$ have pairwise disjoint
interiors and $\cup_{k=1}^m A_k =\Omega$.
Let $\Pi_k$, $k=1,\ldots,m$, denote the half-spaces determined by
$\Omega$ so that 
\[
E_k\subset \partial\Pi_k ,\;\; k=1,\ldots, m,
\quad \mbox{ and} \qquad \Omega=\bigcap_{k=1}^m \Pi_k.
\]
We then have
\be
d_N(\xi) ={\rm dist}_N(\xi , \partial \Pi_k) =:d_{N,k}(\xi)
\; , \qquad \mbox{ for all }\xi\in A_k.
\la{geom}
\ee
It immediately follows from \eqref{geom} that
\[
d_N(\xi) =\min_{1\leq k\leq m} d_{N,k}(\xi),
\qquad \xi\in\Omega.
\]
We fix a non-negative function
$\phi\in\cic(\Omega)$ and we aim to show that
\be
\int_{\Omega}
 \nabla_{H}d_N \cdot 
 \nabla_{H} \phi \, {\rm d}\xi 
=\sum_{k=1}^m  \int_{A_k} \nabla_{H}d_{N,k} \cdot 
 \nabla_{H} \phi \, {\rm d}\xi  \geq 0.
\la{ibp}
\ee
We recall the divergence theorem in the stratified setting:
if the vector field $F$ takes values in the first stratum and
sufficient regularity is assumed then
\[
\int_{A}{\rm div}_{H}F \, {\rm d}\xi =
\int_{\partial A} F\cdot \nu_{H} dS \, ,
\]
where
\be
\nu_{H} =(\nu_x +2 \nu_t y , \nu_y-2\nu_t x)
\la{nu}
\ee
and  $\nu=(\nu_x,\nu_y,\nu_t)$ denotes the usual outer normal 
vector.

We cannot directly apply integration by part to each of the 
integrals in the RHS of \eqref{ibp}
due to the fact that, as is seen
from Lemma \ref{lem:d_Nformula}, there is a halfline $L_k\subset \Pi_k$
on which the function $d_{N,k}$ is not differentiable (the halfline $L_k$ is the image of the halfline $\{ (0,0,t) : t>0\}$
under the group action that maps $\Pi_0$ onto $\Pi_k$).
We define the set
\[
K =  \bigcup_{k=1}^m L_k 
\]
and so $d_N$ is $C^1$ in $\Omega\setminus K$.

We now use a standard cut-off argument.
We denote by $d_E(\cdot,K)$ the Euclidean distance to the set 
$K$ and we consider smooth functions $\psi_{\epsilon}$,
$\epsilon>0$, such that
$0\leq \psi_{\epsilon}\leq 1$ and
\[
\left\{
\begin{array}{ll}
\psi_{\epsilon}(\xi) =0, & \mbox{ if }d_E(\xi, K) 
< \epsilon , \\[0.1cm]
\psi_{\epsilon}(\xi) =1 , & \mbox{ if }d_E(\xi, K)
 > 2\epsilon , \\[0.1cm]
|\nabla\psi_{\epsilon}(\xi)| \leq \frac{c}{\epsilon} ,
& \mbox{ for all }\xi\in\H.
\end{array}
\right.
\]
We then define $\phi_{\epsilon}=\phi \, \psi_{\epsilon}$.

Let $k\in\{1,\ldots,m\}$.
Integrating by parts and using Proposition \ref{PROP:half-space}
we obtain
\begin{align*}
\int_{A_k}
 \nabla_{H}d_{N,k} \cdot \nabla_{H}\phi_{\epsilon} \, {\rm d}\xi 
 = & - \int_{A_k}\phi_{\epsilon} \, \Delta_{H} d_{N,k}
   \, {\rm d}\xi  +
  \int_{\partial A_k} \!\!
\phi_{\epsilon} \, \nabla_{H}d_{N,k} \cdot \nu_{k,H} \, dS \\
\geq & \int_{\partial A_k} \!\!
\phi_{\epsilon} \, 
\nabla_{H}d_{N,k} \cdot \nu_{k,H} \,  dS.
\end{align*}
where $\nu_{k,H}$ is defined as above relative to $A_k$.
Adding over $k$ we arrive at
\be
\int_{\Omega}
 \nabla_{H}d_N \cdot \nabla_{H}\phi_{\epsilon} \, {\rm d}\xi 
\geq \sum_{k=1}^m  \int_{\partial A_k} \!\!
\phi_{\epsilon} \, 
\nabla_{H}d_{N,k} \cdot \nu_{k,H} \,  dS.
\la{ap}
\ee
Now, each boundary $\partial A_k$ consists of outer
parts where $\phi$ vanishes as well as
of common boundaries with other sets $A_j$, $j\neq k$.
Let us fix $k,j$ such a set $A_k$ and $A_j$
share such a common boundary $S_{kj}$.
Let $\nu$ be the normal vector on $S_{kj}$
which is outer with respect to $A_k$ (the existence of $\nu$ is
shown in the Appendix in case where the two hyperplanes are not
vertical; otherwise one may use an approximation argument,
see also \cite{Lar16}).
Defining $\nu_H=(\nu_x +2 \nu_t y , \nu_y-2\nu_t x)$
we conclude that the
two contributions on the surface $S_{kj}$ from $A_k$
and $A_j$ add up to
\[
\int_{S_{kj}} \!\!
\phi_{\epsilon} \, 
(\nabla_{H}d_{N,k} - \nabla_{H}d_{N,j})\cdot \nu_{H} dS.
\]
The surface $S_{kj}$ is a level set for the function
$d_{N,k}-d_{N,j}$ and at each point $\xi\in S_{kj}$
there holds $ \nabla d_{N,k} - \nabla d_{N,j} =\lambda\nu$
where $\lambda=\lambda(\xi)\geq 0$. We therefore
have on $S_{kj}$
\begin{align*}
&(\nabla_{H}\, d_{N,k} - \nabla_{H}\, d_{N,j})\cdot \nu_{H} 
= \big( \nabla_{\! x} \, d_{N,k} - \nabla_{\! x} \, d_{N,j}
 + 2(d_{N,k})_t y -2(d_{N,j})_t y  ,  \\
 & \quad \quad
 \nabla_{\! y}\, d_{N,k} - \nabla_{\! y}\, d_{N,j}
 - 2(d_{N,k})_t x +2(d_{N,j})_t x    \big)   \cdot 
 \big( \nu_x +2y \nu_t , \nu_y-2x\nu_t\big) \\
&= \big(  \lambda \nu_x +2\lambda\nu_t y ,
\lambda \nu_y -2\lambda\nu_t x \big)\cdot
\big( \nu_x +2 \nu_t y , \nu_y-2\nu_t x\big) \\
& = \lambda\Big( (\nu_x +2 \nu_ty)^2 +
(\nu_y  -2\nu_t x)^2\Big) \\
& \geq 0.  
\end{align*}
Hence the LHS of \eqref{ap} is non-negative. Noting that
$\nabla\phi_{\epsilon}\to \nabla\phi$ in $L^1(\Omega)$ as
$\epsilon\to 0$ completes the proof.
Combining the above completes the proof
of (i).
Part (ii) is an immediate consequence of part (i) and 
Theorem \ref{thmm1}. The  optimality of the constant follows from
Theorem   \ref{thm:half-space.d_N}.\endproof

\

As an immediate consequence of Theorem \ref{THM:Hardy,N,Omega} we obtain the geometric uncertainty principle on the convex set $\Omega \subset \H^n$ with respect to the gauge pseudo-distance on $\H^n$.
\begin{corollary}\label{corollary}
Let $D\subset\Hei^n$ be either a bounded convex polytope or an arbitrary half-space in $\H^n$
and let 
\[
d_{N}(\xi)={\rm dist}_{N}(\xi,\partial D)
\]
denote
the corresponding  pseudodistance of $\xi \in D$ to the boundary $\partial D$. Then for any $u \in C_{c}^{\infty}(D)$ we have
\[
\left(\int_{D} |\nabla_{\H^n}u|^2 {\rm d}\xi \right)^{\frac{1}{2}}\left(\int_{D} d_{N}^{2}u^2
{\rm d}\xi\right)^{\frac{1}{2}}\geq \frac{1}{2} \int_{D} u^2 {\rm d}\xi\,.
\]
\end{corollary}
\begin{proof}
    A combination of Theorem \ref{THM:Hardy,N,Omega}, Part (ii) and of the Cauchy-Schwarz inequality yields
    \begin{eqnarray*}
       \left( \int_{D} |\nabla_{\H^n}u|^2 {\rm d}\xi 
       \right)
\left(\int_{D} d_{N}^{2}u^2 {\rm d}\xi \right)
& \geq & \frac{1}{4} 
\left(\int_{D} \frac{u^2}{d_{N}^2}{\rm d}\xi
\right)
\left( \int_{D}d_{N}^2 u^2 {\rm d}\xi\right)\\
 & \geq & \frac{1}{4} \left( \int_{D} u^2 {\rm d}\xi\right)^2.
 \end{eqnarray*}
\end{proof}

\subsection{Hardy inequalities with respect to the Carnot-Carath\'{e}odory distance}

In this section we consider geometric Hardy 
inequalities in the Heisenberg group with respect to 
the Carnot-Carath\'eodory distance. 

In the case of the Heisenberg group $\H^n$, it has been proved in \cite{BGG00,CCG07} (see also
\cite{DLZ24}) that the
Carnot-Carath\'{e}odory distance of a point $\xi=(x,y,t)$
to the origin is given by
\be
\rho(\xi) = 
\left\{
\begin{array}{ll}
\frac{\phi}{\sin\phi} r , & 
\mbox{ if }(x,y)\neq (0,0), \\[0.2cm]
\sqrt{\pi |t|} , & \mbox{ if } (x,y)=(0,0),
\end{array}
\right.
\la{444}
\ee
where $r=\sqrt{|x|^2+|y|^2}$ and
the angle $\phi\in  (0 , \pi)$ is uniquely
determined by the requirement
\be
\la{mu}
\mu(\phi):= \frac{2\phi -\sin(2\phi)}{2\sin^2\phi}=
\frac{t}{r^2}.
\ee

For a given domain $\Omega\subset\Hei^n$ the induced 
distance to the boundary is given by
\be
d_{\rho}(\xi) = {\rm dist}_{\rho}(\xi,  \partial\Omega)
=\inf\{ \rho((\xi')^{-1}\xi) , \; \xi' \in \partial\Omega\},
\qquad \xi\in\Omega.
\la{ntorb}
\ee
As in Section \ref{sec:gauge}, we first consider the case
of the half-space
\[
\Pi_0 =\{ (x,y,t) \in \Hei^n : t>0\}.
\]
By the above, it is then easy to see that for $\xi=(x,y,t)\in \Pi_0$
and $\xi'=(x',y',0)\in\partial \Pi_0$ we have
\be
d_{\rho}(\xi ,\xi') =
\left\{
\begin{array}{ll}
\frac{\phi}{\sin\phi}\sqrt{|x'-x|^2+|y'-y|^2} ,
 & (x',y')\neq (x,y) , \\[0.2cm]
\sqrt{\pi  t } 
\, , & (x',y')= (x,y),
\end{array}
\right.
\la{we}
\ee
where $\phi\in (0,\pi)$ is implicitly given by
\be
\la{phi}
\mu(\phi) =\frac{t +2(x\cdot y'-x'\cdot y)}{|x'-x|^2+|y'-y|^2}.
\ee
We set
\[
g(\phi)=\frac{\phi}{\sin\phi}
\]
and define the function
\[
F(x',y') =g(\phi) \sqrt{|x'-x|^2+|y'-y|^2} ,
\qquad (x',y')\neq (x,y).
\]
where $\phi\in  (0 , \pi)$ is determined by
\eqref{phi} with $\xi\in\Pi_0$ being fixed.

\begin{lemma}
\la{lem2}
Let $\xi=(x,y,t)\in \Pi_0$ with $(x,y)\neq (0,0)$. The function
$F(x',y')$ has a unique critical point $(x',y')$ in the set
$\R^{2n}\setminus \{(x,y)\}$. Moreover
this critical point is given
by
\be
\left\{
\begin{array}{l}
(x_i'-x_i)\cot \phi =y_i \;  , \\ 
(y_i'-y_i)\cot\phi = -x_i  \; , \qquad i=1,\ldots,n,
\end{array}
\right.
\la{tan}
\ee
where $\phi$ is the unique solution in $(0,\pi/2)$ of the equation
\be
\la{phi1}
 \frac{ 2\phi +\sin(2\phi)}{2\cos^2\phi} =
 \frac{t}{r^2}.
\ee
\end{lemma}
\proof
For notational simplicity we set $X=(x,y)$ and
$X'=(x',y')\neq X $. We then have
\be
\la{ol1}
F_{x_i'}(x',y') =g'(\phi)\phi_{x_i'}|X'-X| + 
g(\phi) \frac{x_i'-x_i}{|X'-X|}.
\ee
Differentiating \eqref{phi} we get
\begin{align*}
\mu'(\phi)\phi_{x_i'} = & -\frac{2y_i}{|X'-X|^2} -
\frac{2\big(t +2(x\cdot y'-x'\cdot y)\big))(x_i'-x_i)}{|X'-X|^4} \\
= & \frac{-2y_i - 2\mu(\phi)\, (x_i'-x_i)}{|X'-X|^2}.
\end{align*}
Hence going back to \eqref{ol1} we see that
$F_{x_i'}(x',y')=0$ if and only if
\be
\Big( g(\phi)\mu'(\phi) -2g'(\phi)\mu(\phi)\Big)
(x_i'-x_i) =  2g'(\phi)y_i.
\la{ol2}
\ee
But
\[
\mu'(\phi)=\frac{2(\sin\phi -\phi\cos\phi)}{\sin^3\phi} ,
\qquad g'(\phi) =\frac{\sin\phi -\phi\cos\phi}{\sin^2\phi}
\]
and 
\[
g(\phi)\mu'(\phi) -2g'(\phi)\mu(\phi) = 2\; \frac{\cos \phi
( \sin\phi -t\cos\phi)}{ \sin^3\phi}.
\]
so \eqref{ol2} takes the form $(x_i'-x_i)\cot \phi =y_i$.
Similar considerations show that $F_{y_i'}(x',y')=0$
if and only if $(y_i'-y_i)\cot\phi =-x_i$. Hence we conclude
that a point $(x',y')\neq (x,y)$ is a critical
point of $F$ if and only if \eqref{tan} is satisfied
where $\phi\in (0,\pi)$ is given by \eqref{phi}.

Assume now that $(x',y')\neq (x,y)$ is a critical point of the function $F$.
From \eqref{tan} we then have that $\phi\neq \pi/2$.
Moreover, again from \eqref{tan},
\[
x_i(x_i'-x_i)+y_i(y_i'-y_i)=0 \, , \qquad i=1,\ldots,n.
\]
and
\[
x\cdot y'-x'\cdot y =x\cdot (y'-y)-(x'-x)\cdot y 
=- r^2 \tan\phi .
\]
It then follows that
\be
\la{tan2}
|x'-x|^2+|y'-y|^2 =r^2 \tan^2\phi .
\ee
These together with \eqref{phi} give
\[
\mu(\phi) = \frac{ t-2r^2 \tan\phi}{( r^2 \tan^2\phi}.
\]
Recalling the definition of the function $\mu$, cf. 
\eqref{mu},  we conclude that $\phi$ satisfies \eqref{phi1}.
This procedure defines a map
\[
  \Pi_0 \setminus\{ (0,0,t) : t>0\} \ni (x,y,t) \mapsto \phi=\phi(x,y,t)  \in (0,\pi).
\]
We shall write $\phi=\phi(r,t)$, $r=\sqrt{|x|^2+|y|^2}$.
This map is clearly continuous and, as we have seen, we 
have $\phi(r,t)\neq \pi/2$ for all $r,t>0$. But,
by \eqref{phi1}, $\phi(r,t)\to 0$ as $t\to 0+$ for any fixed
$r>0$, hence by continuity we obtain that in fact
$ \phi(r,t)\in (0,\pi/2)$ for all $r,t>0$.

Hence given $(x,y,t)\in \Pi_0$ with $(x,y)\neq (0,0)$,
if a critical point $(x',y')\neq (x,y)$ exists then it is
unique and it is determined by \eqref{phi1} and
\eqref{tan}.

To complete the proof we observe that the argument works both 
ways: if $\phi\in (0,\pi/2)$ is defined by \eqref{phi1} then 
the point $(x',y')$ defined by \eqref{tan} is different
from $(x,y)$ and is indeed a critical point of $F$.
\endproof

\begin{lemma}
\la{thm2}
The Carnot-Carath\'{e}odory distance of a point
$\xi=(x,y,t)\in \Pi_0$ to the boundary $\partial\Pi_0$
is given by
\[
d_\rho(\xi) = \left\{
\begin{array}{ll}
\frac{\phi}{\cos\phi} r , &
\mbox{ if }(x,y)\neq (0,0) ,\\[0.2cm]
\sqrt{\frac{\pi t}{2}}, & \mbox{ if }(x,y)= (0,0),
\end{array}
\right.
\]
where $r=\sqrt{|x|^2+|y|^2}$ and
in the first case $\phi\in (0,\pi/2)$ is uniquely 
determined by the requirement
\be
\la{phi10}
 \frac{ 2\phi +\sin(2\phi)}{2\cos^2\phi} =\frac{t}{r^2}.
\ee
Moreover in case $(x,y)\neq (0,0)$ the distance is realized
at a unique point $\xi'\in\partial\Pi_0$ while in case
$(x,y)=(0,0)$ it is realized at all points of the circle
with center at the origin and radius $\sqrt{2 t/\pi}$.
\end{lemma}
\proof 
(i) {\em Case $(x,y)\neq (0,0)$.}
Let $\xi=(x,y,t)\in\Pi_0$ with $(x,y)\neq (0,0)$ be given
and let $(x',y')\neq (x,y)$ be the critical point of $F$
determined in Lemma \ref{lem2}. Writing $\xi'=(x',y',0)$
and recalling \eqref{we} and \eqref{tan2} we have
\begin{align}
d_{\rho}(\xi,\xi') =& F(x',y') \nonumber \\
= & \frac{\phi}{\sin\phi} \sqrt{|x'-x|^2+|y'-y|^2} \nonumber\\
= & \frac{\phi}{\sin\phi} r \; \tan\phi \nonumber\\
= & \frac{\phi}{\cos\phi} r .
\la{kron}
\end{align}
The proof will be complete if we prove that
the distance of $\xi$ to the point
$(x,y,0)$  is strictly larger
than $d_{\rho}(\xi , \xi')$.
To see this we recall that, by \eqref{we},
\[
d_{\rho}\big(  \xi \, , \, (x,y,0) \big) =\sqrt{\pi t}.
\]
Applying the elementary inequality
\[
\phi^2 <\frac{\pi}{2} \big(  2\phi +\sin(2\phi) \big) \; ,
\qquad 0<\phi<\frac{\pi}{2}.
\]
we thus have
\begin{align*}
d_{\rho}^2(\xi,\xi') & =\frac{\phi^2}{\cos^2\phi} r^2 \\
&< \frac{2\phi +\sin(2\phi)}{2 \cos^2\phi} \pi r^2 \\
& = \frac{t}{r^2} \, \pi r^2 \\
& = d_{\rho}^2\big(  \xi \, , \, (x,y,0) \big) .
\end{align*}

\noindent
(ii) {\em Case $(x,y)= (0,0)$.}
Let $t>0$ be fixed. We recall that
\[
d_{\rho}\big( (0,0,t) \, , (0,0,0)  \big) =\sqrt{\pi t}.
\]
Now, for $(x',y')\neq (0,0)$ we define
\[
G(x',y')=d_{\rho}^2\big(  (x',y',0) \, , \, (0,0,t) \big)
= \frac{\phi^2}{\sin^2\phi}r'^2
\]
where $r'^2 =x'^2+y'^2$ and $\phi\in (0,\pi)$ is defined by
\[
\frac{2\phi -\sin(2\phi)}{2\sin^2\phi}=
\frac{t}{r'^2}.
\]
Hence
\[
G(x',y')  = \frac{2\phi^2}{ 2\phi -\sin(2\phi)}t.
\]
The function $\phi\mapsto \phi^2/(2\phi -\sin(2\phi))$
is minimized for $\phi=\pi/2$ in which case
it is equal to $\pi /2$. Hence $d_{\rho}(0,0,t)=\sqrt{\pi t/2}$.
Moreover we have $\phi=\pi/2$
precisely for the points $(x',y')$ for which
$r'^2=2t/\pi$. This completes the proof.
\endproof

\begin{proposition}\label{prop:drho}
Let $\Pi_0=\{(x,y,t) \in \H^n, t>0\}$ and let $d_\rho=d_\rho(\xi)$, $\xi \in \Pi_0$, denote the corresponding Carnot-Carath\'eodory distance to the boundary $\partial \Pi_0$. Then for any fixed $r \neq 0$ we have 
\[
d_\rho(\xi)=\frac{t}{2r}+o(t)\,,\quad  \mbox{as $t \rightarrow 0^{+}$. }
\]
\end{proposition}
\begin{proof}
  The proof  follows 
from Lemma \ref{thm2} using standard arguments, since by \eqref{phi10} $t \rightarrow 0$ implies that $\phi \rightarrow 0$.
\end{proof}

\begin{theorem}
\label{THM:Hardy_halfsp}
Let $\Pi$ be an arbitrary half-space in $\H^n$
and let $d_\rho(\xi)$, $\xi \in \H^n$,
denote the Carnot-Carath\'{e}odory distance to
the  boundary $\partial\Pi$. Then 
\begin{align*}
{\rm (i)} & \quad  \Delta_{\H^n} d_\rho \leq 0 \; \mbox{ in the distributional
sense in }\Pi .\\
{\rm (ii)} & \quad \mbox{For any $p>1$ the Hardy inequality}
\\[0.2cm]
& \hspace{1.5cm}
 \int_{\Pi} |\nabla_{\H^n} u|^p {\rm d}\xi \geq 
 \Big(\frac{p-1}{p}\Big)^p
\int_{\Pi} \frac{|u|^p}{d_{\rho}^{p}} {\rm d}\xi  \; ,  \quad u\in \cic(\Pi), \\[0.2cm]
& \quad \mbox{is valid. Moreover the constant is the best possible.}
\end{align*}
\end{theorem}
\proof $\ia$ We first note that by invariance under group action we may assume
that $\Pi=\Pi_0$. 
By Lemma  \ref{thm2} the distance to the boundary of a point
$\xi=(x,y,t)$, $(x,y)\neq (0,0)$, is given by
\[
d_{\rho}(r,t) = \frac{\phi}{\cos\phi} r =: B(\phi)r.
\]
We will see below that $d_{\rho,r}(r,t)\leq 0$. Therefore, we have
\begin{align}
\Delta_{\H^n} d_{\rho}= & \;  d_{\rho,rr} +\frac{2n-1}{r} d_{\rho,r} + 4r^2d_{\rho,tt}  
\label{ol}\\  
\leq & \;  d_{\rho,rr} +\frac{1}{r} d_{\rho,r} + 4r^2d_{\rho,tt}  \nonumber \\  
= & \; B''(\phi) \Big(  r\phi_r^2 +4r^3 \phi_t^2 \Big)   +B'(\phi) \Big(  r\phi_{rr}  +3\phi_r +4r^3 \phi_{tt}\Big) 
+ \frac{1}{r}B(\phi).
\nonumber
\end{align}
The various partial derivatives of $\phi=\phi(r,t)$
are computed from the relation
\be
\la{phi2}
Q(\phi):= \frac{ 2\phi +\sin(2\phi)}{2\cos^2\phi} =\frac{t}{r^2}.
\ee
Differentiating we find
\[
\phi_t =\frac{1}{r^2 Q'(\phi)} \; , \qquad
\phi_{tt} = -\frac{Q''(\phi)}{r^4 \, (Q'(\phi))^3}
\; , \qquad
\phi_r =-\frac{2Q(\phi)}{r\, Q'(\phi)} \; ,
\]
\[
\phi_{rr} = \frac{Q(\phi)}{r^2 (Q'(\phi))^3}
\Big(  6\, (Q'(\phi))^2 -4 Q''(\phi) \, Q(\phi)\Big).
\]
Substituting in \eqref{ol} we arrive at
\be
\la{for}
\Delta_{\H^n} d_{\rho} \leq \frac{1}{r(Q')^3} \bigg(
4(Q^2+1)(B''Q'-Q''B')  +B\, (Q')^3
\bigg).
\ee
Let $A$ denote the term in large brackets in \eqref{for}.
We have
\[
B'(\phi) = \frac{ \cos \phi +\phi \, \sin \phi}{ \cos^2 \phi} \; , \qquad
\quad B''(\phi) =\frac{ \phi+ \phi\sin^2 \phi +\sin(2\phi)}{ \cos^3\phi},
\]
\[
Q'(\phi)= 2\, \frac{\cos \phi+ \phi\, \sin \phi}{\cos^3 \phi}
 \; , \qquad \quad
Q''(\phi)= 2\, \frac{\phi +2\phi \sin^2 \phi +3\sin(2\phi)}{\cos^4 \phi}.
\]
Therefore
\[
d_{\rho,r}(r,t) =B(\phi) +B'(\phi)\phi_r \, r =-\sin\phi \leq 0
\]
as claimed.
Substituting we find after some computations that
\begin{align*}
A= & -\frac{8}{\cos^9 \phi} \cdot \bigg(
(1+ 2 \cos^4 \phi -3 \cos^2 \phi )\phi^3 +
\cos \phi \sin \phi \, (3 -5 \cos^2 \phi )\phi^2 \\
& \qquad\qquad \qquad + \cos^2 \phi \, (3 -4 \cos^2 \phi)\phi+\cos^3 \phi \sin \phi \bigg).
\end{align*}
Now, the term in the large brackets above can be written as
\begin{align*}
& \phi^3 \sin^4\phi + \phi^2 \sin^2\phi \,  \cos\phi \, (3\sin\phi -\phi\cos\phi) \\
& \qquad \quad  +\phi \sin\phi \cos^2\phi \, (3\sin\phi -2\phi \cos\phi)
 +\cos^3 \phi \, (\sin\phi -\phi\cos\phi)
\end{align*}
Each of these four terms is non-negative for 
$\phi\in (0,\pi/2)$; hence $A\leq 0$ and (i) has been
proved.

\

\noindent
(ii) The required Hardy inequality will follow by applying
Theorem \ref{THM:Hardy,gen} provided we establish
that $\Delta_{p,\H^n}d_{\rho} \leq  0$
in the distributional sense in $\Pi$.
Now, by \cite[Theorem 3.1]{MSC01} (or by a direct computation
for $r\neq 0$) we have $|\nabla_{\H^n} d_{\rho}|=1$ a.e.

Hence the above condition can be simplified to
$\Delta_{\H^n}d_{\rho} \leq  0$ which is in particular independent of 
$p>1$; see also \cite{BFT04}.
So the result follows from (i). The sharpness of the constant follows from Theorem \ref{THM:Hardy,gen} (b),  Proposition \ref{prop:drho} and the fact that $|\nabla_{\H^n} d_{\rho}|=1$ a.e.,
cf. \cite{MSC01}.
\endproof

The next proposition provides a more detailed picture
concerning nearest boundary points. We denote $r'^2=|x'|^2+|y'|^2$.

\begin{proposition}
$(1a)$ Any point $\xi=(x,y,t)\in \Pi_0$ with $(x,y)\neq (0,0)$
has a unique nearest boundary
point $(x',y',0)\in \partial\Pi_0$. Moreover $(x',y')$
is different from $(x,y)$ and from $(0,0)$ and
there holds $d_\rho(\xi) < \pi r'/2$.

\noindent
$(1b)$ Let $t>0$. The point $(0,0,t)$ has as nearest boundary
points all points $(x',y',0)$ with $|x'|^2+|y'|^2=2t/\pi$.

\noindent
$(2a)$ Conversely, let $\xi'=(x',y',0)\in \partial\Pi_0$
with $(x',y') \neq (0,0)$ and $\rho>0$ be given. 
If $\rho< \pi r'/2$ then there are exactly two points
$\xi=(x,y,t)\in\Pi_0$ whose nearest boundary point is
$\xi'$ and for which $d_\rho(\xi)=\rho$. Moreover exactly one
of these points lies on the $t$-axis. 
If $\rho\geq  \pi r'/2$ then there is only one such point
and it lies on the $t$-axis.

\noindent
$(2b)$ No point in $\Pi_0$ has $(0,0,0)$ as its
nearest boundary point.
\la{lem:coord2}
\end{proposition}
\proof (1a) We have already seen in the proof
of Theorem \ref{thm2} that given $\xi\in\Pi_0$ the nearest 
boundary point $\xi'=(x',y',0) \in\partial\Pi_0$ is uniquely defined. Moreover, setting $r'^2 =x'^2 +y'^2$ we have from
\eqref{tan} that
\[
r'^2 =(1+\tan^2\phi)\, r^2 =\frac{r^2}{\cos^2\phi}
\]
where $\phi\in (0,\pi/2)$ is given by \eqref{phi10}.
Hence
\be
d_{\rho}(\xi) =\frac{\phi}{\cos\phi} r =\phi r' <\frac{\pi r'}{2}.
\la{we1}
\ee

\noindent (1b) This is contained in Theorem \ref{thm2}.

\noindent (2a)
Assume that $\xi=(x,y,t)\in \Pi_0$ with $(x,y)\neq (x',y')$
has $\xi'$
as its nearest boundary point and that $d_{\rho}(\xi)=\rho$.
By \eqref{we1} $\phi =\rho/r'$ and the point $\xi$ is 
now uniquely determined by \eqref{tan} and \eqref{phi10}.
It is now easy to
see that this point $\xi$
has indeed $\xi'$ as its nearest boundary point and
$d_{\rho}(\xi)=\rho$.

Moreover, by Theorem \ref{thm2}, the point
$(0,0,2\rho^2/\pi)$ also has $(x',y',0)$
as one of its nearest 
boundary points and its distance to the boundary is $\rho$.
In case $\rho\geq \pi r'/2$ the first of these two points is
not defined. Hence (2a) has been proved. 

\noindent (2b) We note that in case (1a) the nearest boundary
point $(x',y',0)$ is not the origin since $(0,0)$
is not a critical point of the function $F$ in Lemma 
\ref{lem2}. In case (1b) the nearest boundary points
are also different from the origin.

\endproof

\begin{lemma}
\label{lem:hardy_polytope}
Let $\Omega\subset\Hei^n$ be a bounded convex
polytope. The corresponding Carnot-Carath\'{e}odory
distance to the  boundary satisfies
$\Delta_{\H^n} d_\rho \leq 0 $ in the distributional
sense in $\Omega$.
\end{lemma}
\begin{proof}
    The proof follows exactly the lines of the proof of Theorem \ref{THM:Hardy,N,Omega} using Lemma \ref{thm2} and Theorem \ref{THM:Hardy_halfsp} instead.
\end{proof}

\begin{theorem}
\label{thm:hardy_convex}
Let $\Omega\subset\Hei^n$ be bounded and convex
and let $d_\rho(\xi)$, $\xi \in \Omega$,
denote the corresponding Carnot-Carath\'{e}odory
distance to the  boundary. Then 
for any $p>1$ the Hardy inequality
\[
 \int_{\Omega} |\nabla_{\H^n} u|^p {\rm d}\xi \geq 
 \Big(\frac{p-1}{p}\Big)^p
\int_{\Omega} \frac{|u|^p}{d_{\rho}^{p}} {\rm d}\xi  \; ,  \quad u\in \cic(\Omega),
\]
is valid.
\end{theorem}
\proof Let  $u\in\cic(\Omega)$ be fixed. 
We consider a bounded convex polytope $\Omega'$ 
such that
\[
 {\rm supp}(u) \subset\subset \Omega' \subset\subset \Omega.
\]
Let us denote  by $d_{\rho}'$ the Carnot-Carath\'{e}odory
distance to $\partial\Omega'$. 
Combining Theorem \ref{THM:Hardy,gen} and Lemma \ref{lem:hardy_polytope}
we obtain that
\[
 \int_{\Omega'} |\nabla_{\H^n} u|^p {\rm d}\xi \geq 
 \Big(\frac{p-1}{p}\Big)^p
\int_{\Omega'} \frac{|u|^p}{d_{\rho}'^p} {\rm d}\xi  \; .
\]
Hence, since $d_{\rho}'\leq d_{\rho}$ in $\Omega'$,
\[
\int_{\Omega} |\nabla_{\H^n} u|^p {\rm d}\xi 
=  \int_{\Omega'} |\nabla_{\H^n} u|^p {\rm d}\xi 
\geq  \Big(\frac{p-1}{p}\Big)^p
\int_{\Omega'} \frac{|u|^p}{d_{\rho}'^p} {\rm d}\xi
\geq  \Big(\frac{p-1}{p}\Big)^p
\int_{\Omega} \frac{|u|^p}{d_{\rho}^p} {\rm d}\xi
\]
and the proof is complete. \endproof

\begin{remark}\label{rem:d<d'}
We point that the fact that the key property that
$d_{\rho}'\leq d_{\rho}$,
for $d$ and $d'$ as in the proof of Theorem \ref{thm:hardy_convex} in the bounded convex polytope $\Omega'$ and reflects the geometric nature of the Carnot-Carath\'eodory distance, meaning in particular that that the latter is a ``true'' distance respecting the geometry of $\H^n$. 
\end{remark}

The proof of the following corollary, which is the geometric uncertainty principle with respect to the Carnot-Carath\'eodory distance, follows the lines of Corollary \ref{corollary} in the case of the gauge pseudodistance.

\begin{corollary}
Let $D\subset\Hei^n$ be either a bounded convex domain or an arbitrary half-space in $\H^n$
and let 
\[
d_{\rho}(\xi)={\rm dist}_{\rho}(\xi,\partial D)
\]
denote
the corresponding Carnot-Carath\'eodory  distance of $\xi \in D$ to the boundary $\partial D$. Then for any $u \in C_{c}^{\infty}(D)$ we have
\[
\left(\int_{D} |\nabla_{\H^n}u|^2 {\rm d}\xi \right)^{\frac{1}{2}}\left(\int_{D} d_{\rho}^{2}u^2
{\rm d}\xi\right)^{\frac{1}{2}}\geq \frac{1}{2} \int_{D} u^2 {\rm d}\xi\,.
\]
\end{corollary}

\section{Hardy inequalities on stratified groups of step two}
\label{sec:convexity}

In this section we consider stratified groups of step two. If $G\equiv\R^n$ is such a group with (cf. \eqref{def.carnot})
$\text{dim}(V_1)=m<n$, then each element $g\in G$
can be written as
\[
g= (g^{(1)},g^{(2)})=
(g_1,\cdots,g_m,g_{m+1},\cdots,g_n),
\]
where $g^{(1)} \in \R^{m}$ and $g^{(2)}\in \R^{n-m}$ 
belong in the first and the second stratum of $G$, 
respectively. It is known
that the group law has the form
     \begin{equation}
     \label{g.law1}
        (g'g)_i =\left\{
 \begin{array}{ll}
g_i+g'_i\,, & i=1,\cdots,m,\\
g_i+g'_i+ \frac12 \langle B^{(i)}g'^{(1)},g^{(1)} \rangle\, ,
& i=m+1,\cdots,n ,
 \end{array}
 \right.\, \end{equation}
  where the $B^{(i)}$'s are $m \times m$ matrices, and $\langle \cdot, \cdot \rangle$ stands for
  the standard inner product in $\R^m$, see e.g. \cite[Remark 17.3.1]{BLU07}. The group law \eqref{g.law1} can also be written as 
  \begin{equation}\label{g.law2}
    g'g=(g^{(1)}+g'^{(1)}, g^{(2)}+g'^{(2)}+
    \frac12 \langle Bg'^{(1)}, g^{(1)}\rangle)\,,  
  \end{equation}
  where   $\langle Bg^{(1)},g'^{(1)} \rangle$ denotes the $(n-m)$-tuple 
  \[
  (\langle B^{(m+1)}g^{(1)},g'^{(1)} \rangle, \cdots, \langle B^{(n)}g^{(1)},g'^{(1)} \rangle)\,.
  \]
The inverse element is then given by
\[
 (g^{(1)},g^{(2)})^{-1}=(-g^{(1)},-g^{(2)}+
  \frac12 \langle Bg^{(1)},g^{(1)} \rangle)\, .
\]

We note that the (anisotropic) dilations on a stratified group $G$ of step two are given
by the maps $\delta_\lambda$, $\lambda >0$, defined by
  \[
  \delta_\lambda((g^{(1)},g^{(2)}))=(\lambda g^{(1)}, \lambda^2 g^{(2)}) .
  \]
 
In the first part of this section we prove some results on the concavity, in the 
sense of \cite{DGN03, LMS03}, of the Euclidean distance to the boundary on a 
convex set $\Omega \subset G$. This, combined with results for \cite{DGN03}, 
yields Theorem C of the introduction and in particular
the $L^2$ Hardy inequality on convex domains $\Omega\subset G$.

\subsection{On the distance function from the boundary of bounded convex domains in
stratified groups}\label{sec:Hg}

To develop the subsequent analysis we first need to clarify the notions of convexity of sets and functions in the stratified setting.
Even though, as mentioned above, these notions were introduced at the same time in \cite{DGN03} and in \cite{LMS03}, here we adopt the notation of \cite{DGN03} since in \cite{LMS03} these notions are developed in the viscosity sense, while for us the weak sense is more suitable. To this end let us first introduce the following auxiliary notion. 

Let $G$ be a stratified group with $\text{dim}(G)=n$ and $\text{dim}(V_1)=m$. Given a point $g \in G$ the horizontal plane $H_{g}$ passing through $g$ is defined by
 \[
 H_{g}=L_{g} (\exp (V_1 \oplus \{0\}))\,,
 \]
where $L_{g}$ denotes the left translation by
$g\in G$ and $\exp: \mathfrak{g} \rightarrow  G$
is the exponential map for the group $G$. In particular, we have that 
 \[
 H_{e}= \exp (V_1 \oplus \{0\})\,,
 \]
where $e \in G$ is the identity element of  $G$.

Following \cite{DGN03}, for
given $g,g' \in G$ and $\lambda \in [0,1]$
we denote by $g_\lambda$  the anisotropic analogue of the standard Euclidean convex combination, that is
\[
  g_{\lambda}=g_{\lambda}(g;g'):=g \delta_{\lambda}(g^{-1}g').
\]

The following definition was given in \cite[Definition 5.5]{DGN03}. 
\begin{definition}
\label{def:H-convex.u}
     A function $u: G \rightarrow (-\infty,\infty]$ is called weakly $H$-convex if $\{g \in G : u(g)=\infty\} \neq G$, and if for every $g \in G$
and $g' \in H_g$ one has
    \[
    u(g_{\lambda})\leq u(g)+\lambda(u(g')-u(g))\,,
    \qquad  \lambda \in [0,1].
    \]
\end{definition}
The notion of a weakly $H$-concave function can be defined accordingly.

In \cite[Definition 7.1]{DGN03} the authors introduced the following definition of convexity of sets in the stratified setting. 

\begin{definition}
A subset $\Omega\subset G$ of a stratified group $G$ is called weakly $H$-convex if for any $g \in \Omega$ and for any $g' \in \Omega \cap H_g$ one has $g_\lambda \in \Omega$ for every $\lambda \in [0,1].$ 
\end{definition}
  
\begin{remark}
It is easy to prove that if $\Omega \subset  \R^n$ is convex in the Euclidean sense then $\Omega$ is a weakly $H$-convex set in a
stratified group $G \equiv \R^n$ of step two. 
To see this we first observe that
by the identification
   \[
    \exp(g_1X_1+\cdots+g_nX_n)=(g_1,\cdots,g_n)\,,
    \]
between $G$ and the corresponding Lie algebra
$\mathfrak{g}$ via the exponential map, we have
$g \in H_e$ if and only if $g$ is of the 
form $g=(g_1, \cdots,g_m,0,\cdots,0)$.
Since $H_g=L_gH_e$, we obtain from \eqref{g.law2}
that $g'\in H_g$ if and only if $g'$ is of the form
  \begin{eqnarray}\label{g'}
  g'=  (g'^{(1)},g'^{(2)}) = (g^{(1)}+v^{(1)}, g^{(2)}+ \frac12 \langle B g^{(1)},v^{(1)}\rangle)
  \end{eqnarray} 
for some $v^{(1)}\in\R^m$.
Suppose now that
$g\in \Omega$ and let $g'\in \Omega \cap H_g$. Then
$g^{-1}g' \in H_e$, which in turn implies that
    \[
    \delta_\lambda(g^{-1}g')=(\lambda(g'^{(1)}-g^{(1)}),0)\,.
    \]
So
\begin{eqnarray*}
         g_\lambda & = & (g^{(1)}+\lambda (g'^{(1)}-g^{(1)}),g^{(2)}+\frac12 \langle Bg^{(1)}, \lambda(g'^{(1)}-g^{(1)})) \\
         & = & (g^{(1)}+\lambda (g'^{(1)}-g^{(1)}),g^{(2)}+\frac12 \lambda \langle Bg^{(1)}, v^{(1)}\rangle \, )
\end{eqnarray*}
    since by \eqref{g'} we have $g'^{(1)}=g^{(1)}+v^{(1)}$, for some $v \in H_e$. Using \eqref{g'} we
    conclude that
\be
 g_\lambda= (1-\lambda)g + \lambda g' \, .
\la{kak}
\ee
Hence if $\Omega$ is convex (in the Euclidean sense) it is also
weakly $H$-convex.
\end{remark}

In the following example we
show that the distance to a hyperplane
with respect to the quasi-norm \eqref{gauge_norm} is not
weakly $H$-concave.
\begin{example}
Let $\mathbb{H}^n$ be the
Heisenberg group and
$\Pi_0=\{(x,y,t)\in\Hei^n : t>0\}$.
We shall prove that the the distance $d_N$
(cf. \eqref{ntor})
is not weakly $H$-concave.

\     
Actually, we shall show that the weak $H$-concavity fails in a neighbourhood
of any boundary point.
Indeed, let $\xi=(x,y,t) \in \Pi_0$,  and for fixed $\alpha>0$ 
let $\xi'=(x',y',t)=(ax,ay,t).$  Then
$ \xi^{-1}\xi'=(x'-x,y'-y,0) \in H_e$,
hence $\xi' \in H_{\xi}$.
Given $\lambda\in (0,1)$ we have by \eqref{kak}
\[
\xi_{\lambda} =  
\big(  \lambda x' +(1-\lambda) x , \; 
 \lambda  y' +(1-\lambda) y , \; t\big).
\]
We use cylindrical coordinates (cf. Section \ref{Sec3})
and write
\[
\xi=(r,\omega,t) , \quad \xi'=(r',\omega,t), \quad
\xi_{\lambda}=(r_{\lambda},\omega,t)
\]
Assume now for contradiction that $d_N$ is
weakly $H$-convex. Then
\[
d_N(r_{\lambda},t) \geq 
(1-\lambda)d_N(r,t) +\lambda d_N(r',t).
\]
Using the asymptotics of Proposition \ref{propp3} we then have
\[
\frac{t}{2r_{\lambda}} \geq 
(1-\lambda)\frac{t}{2r} +\lambda \frac{t}{2r'} +
O(t^3) , \qquad \mbox{ as }t\to 0+.
\]
Hence
\[
\frac{1}{r_{\lambda}} \geq 
(1-\lambda)\frac{1}{r} +\lambda \frac{1}{r'} \; ,
\]
which contradicts the strict convexity of the function
$1/r$.
We
note that the above argument can be implemented
in a small neighbourood of any boundary point
$\xi_0\in\partial\Pi_0$.

    \end{example}

   \subsection{Hardy inequalities with respect to the Euclidean distance}
In this section we prove that the Euclidean distance to the boundary on
a convex, bounded domain $\Omega$ is weakly $H$-concave and superharmonic.
This provides a proof of the $L^p$-Hardy inequality
for such domains.

\begin{theorem}\label{THM:s.concave2}
   Let $G$ be a stratified group of step two and let  $\Omega \subset G$ be a convex, in the Euclidean sense, bounded  domain in $G$.
   Then the Euclidean distance to the boundary
  is a weakly $H$-concave function on $\Omega$.
\end{theorem}
\begin{proof}
Let $\Omega$ be as in the hypothesis and let $g,g' \in \Omega$, with $g' \in H_g$. We want to show that for any $\lambda \in [0,1]$ we have 
    \begin{equation}
        \label{conc.d1}
        d(g \delta_{\lambda}(g^{-1}g'))\geq (1-\lambda)d(g)+\lambda d(g')\,.
    \end{equation}
Notice that showing $B_{r_{\lambda}}(g_\lambda) \subset \Omega$, where $B_{r_{\lambda}}(g_\lambda)$ is the Euclidean ball of radius $r_\lambda=(1-\lambda)
d(g)+\lambda d(g')$ centered at $g_\lambda$, we would have the desired inequality \eqref{conc.d1}. Let  $h \in B_{r_{\lambda}}(g_\lambda)$. Then $|h-g_\lambda|=\rho\leq r_\lambda$. We define
\[
v=\frac{h-g_\lambda}{\rho}\,  , \qquad
g_1=g+\rho_1 v  \,  , \qquad 
g'_1=g'+\rho_2 v,
\]
where 
    \[
    \rho_1:=\frac{d(g)}{(1-\lambda)d(g)+\lambda d(g')}\rho\,, \quad \text{and}\quad \rho_2:=\frac{d(g')}{(1-\lambda)d(g)+\lambda d(g')}\rho\,.
    \]
Then $g_1 \in B_{d(g)}(g)\subset \Omega$
and $g'_1 \in B_{d(g')}(g')\subset \Omega$,
since $|v|=1$, $\rho_1\leq d(g)$, and $\rho_2 \leq d(g')$. Recalling also \eqref{kak} we then have   
    \begin{eqnarray*}
        (1-\lambda)g_1+\lambda g'_1& = & (1-\lambda)(g+\rho_1 v)+\lambda(g'+\rho_2v)\\
        & = & (1-\lambda)g+\lambda g'+ (1-\lambda) \rho_1 v +\lambda \rho_2 v\\
        & = & g_\lambda+\rho v \\
        & = & h\,,
    \end{eqnarray*}
 where the last inequality follows by the choice of $v$. Hence, by the Euclidean convexity of $\Omega$, we have $h \in \Omega$, and the proof is complete.
\end{proof}

From Theorem
  \ref{THM:s.concave2} we 
obtain the following result;
the sharpness of the constant follows under the 
hypotheses of part (b) of Theorem
\ref{THM:Hardy,gen}.

\begin{theorem}
Let $G$ be a stratified group of step two
and let $\Omega\subset G$ be a bounded domain which is 
convex in the Euclidean sense. Then for any $p>1$ we have
\begin{align*}
{\rm (i)} & \quad  \Delta_{p,H} d \leq 0 \; \mbox{ in the distributional
sense in }\Omega ;\\
{\rm (ii)} & \quad \mbox{The Hardy inequality} \\[0.2cm]
& \hspace{1.5cm}\int_{\Omega} |\nabla_{H} u|^p {\rm d}g \geq 
\Big( \frac{p-1}{p}\Big)^p
\int_{\Omega} \frac{|\nabla_{H} 
d|^p}{d^p}|u|^p {\rm d}g \; , 
\quad u\in \cic(\Omega), \\[0.2cm]
& \quad \mbox{is valid.}
\end{align*}
\end{theorem}
\proof 
    Let $\rho\in C_c^\infty(G)$ with $\rho\ge 0$ and 
    $\int_G\rho\,dx=1$. For $\epsilon>0$ we define 
\[
\rho_\varepsilon(x)
= \varepsilon^{-Q}\rho(\delta_{\varepsilon^{-1}}x)\,.
\]
and
\[
d_\varepsilon(x)=(d * \rho_\varepsilon)(x)=
\int_G d(xy^{-1})\rho_\varepsilon(y)dy\,.
\]
We will  show that $d_\varepsilon$ is weakly $H$-concave on the set
\[
\Omega_\varepsilon:=\{x \in \Omega\,:\, B_{CC}(x,\varepsilon)\subset \Omega \}\,.
\]
Let $g\in\Omega_\varepsilon$ and let $g'\in H_g\cap\Omega_{\epsilon}$. Since $g'\in H_g$ there exists $v\in V_1$
such that
\[
g' = g\exp(v) \in \Omega_\varepsilon,
\]
and for every $\lambda \in [0,1]$ we have 
\[
g_\lambda := g\,\delta_\lambda(g^{-1}g') = g\exp(\lambda v)\,,\qquad 
\]
implying that the weak $H$-concavity condition can also be written as 
\[
d(g\exp(\lambda v))
\ge (1-\lambda)d(g) + \lambda d(g\exp v),\qquad \lambda\in[0,1]\,,
\]
For $y \in \text{supp}\rho_\varepsilon \subset B_{CC}(e, \varepsilon)$, we define the curve 
\[
\gamma_y(t) := g\exp(tv)y^{-1},\qquad t\in[0,1]\,.
\]
Then $\gamma_y$ is a horizontal curve inside $\Omega$ connecting $gy^{-1} \in B_{CC}(g,\varepsilon)\subset\Omega$ with $g'y^{-1} B_{CC}(g',\varepsilon)\subset\Omega$. So
\[
d(\gamma_y(\lambda))\geq (1-\lambda)d(\gamma_y(0))+\lambda d(\gamma_y(1))\,,
\]
that is 
\[
d(g\exp(\lambda v)y^{-1})
\ge (1-\lambda)d(gy^{-1})
  + \lambda d(g\exp(v)y^{-1})\, .
\]
Multiplying by $\rho_\varepsilon(y)\ge0$ and integrating we obtain
\[
d_\varepsilon(g\exp(\lambda v))
\ge (1-\lambda)d_\varepsilon(g)
  + \lambda d_\varepsilon(g\exp v)\,,
\]
i.e. $d_\varepsilon$ is weakly $H$-concave in $\Omega_\varepsilon$.

Now, we have
\begin{align}
\Delta_{p,H} d_\varepsilon &= \operatorname{div}_H
\bigl(|\nabla_H d_\varepsilon|^{p-2}\nabla_H d_\varepsilon\bigr)
\nonumber \\
 &=|\nabla_H d_\varepsilon|^{p-4}
 \Big( 
|\nabla_H d_\varepsilon|^{2}\Delta_H d_\varepsilon
+ (p-2)\bigl\langle D^2_{H}d_\varepsilon\,\nabla_H 
d_\varepsilon,\nabla_H d_\varepsilon\bigr\rangle \Big),
\la{aa7}
\end{align}
where
\[
D^2_{H}d_\varepsilon
:= 
\left(\frac12\bigl(X_i X_j d_\varepsilon + X_j X_i 
d_\varepsilon\bigr)\right)_{i,j=1}^{2n} 
\]
is the symmetrized horizontal Hessian of $d_{\epsilon}$
(here $X_{n+k}=Y_k$).
By \cite[Theorem 5.12]{DGN03}  $D_{H}^{2}d_\varepsilon$ is negative 
semidefinite on $\Omega_\varepsilon$. Using also the elementary
matrix inequality $A\leq {\rm tr}(A)I_d$ which is valid from
any symmetric positive definite matrix, we conclude from
\eqref{aa7} that
\[
\Delta_{p,H}d_\varepsilon \leq 0\,,\qquad \text{pointwise on } \Omega_\varepsilon\,.
\]
Finally, to show that $\Delta_{p,H}d \leq 0$ in the weak sense
we need to show that for any
$\varphi \in C_{c}^{\infty}(\Omega)$
\be
\la{claim4}
    \lim_{\varepsilon \to 0} \int_{\Omega} |\nabla_H d_\varepsilon|^{p-2} \langle \nabla_H d_\varepsilon, \nabla_H \varphi \rangle dx= \int_{\Omega} |\nabla_H d |^{p-2} \langle \nabla_H d, \nabla_H \varphi \rangle dx\,.
\ee
Since 
\[
\nabla_H d_\varepsilon \to \nabla_H d\,,\qquad \text{a.e. in supp}(\varphi)\,,
\]
we have
\[
\langle |\nabla_H d_\varepsilon|^{p-2}\nabla_H d_\varepsilon, \nabla_H \varphi\rangle \to \langle |\nabla_H d|^{p-2}\nabla_H d\, \nabla_H \varphi \rangle
\]
Now, notice that since $|\nabla_h d_\varepsilon|$ is bounded on ${\rm supp}(\varphi)$ we have 
\[
\left| \langle |\nabla_H d_\varepsilon|^{p-2}\nabla_H d_\varepsilon, \nabla_H \varphi \rangle \right| \leq C_\varphi\,,
\]
and we can apply the dominated convergence theorem to get 
\[
\lim_{\varepsilon\to 0}\langle \Delta_{p,H} d_\varepsilon,\varphi\rangle = -\int_\Omega |\nabla_H d|^{p-2} \langle \nabla_H d,\nabla_H\varphi\rangle\,dg = \langle \Delta_{p,H} d,\varphi\rangle\,,
\]
which completes the proof of (i) Part (ii) then follows
from Theorem \ref{THM:Hardy,gen}.

\

\section*{Appendix}

We prove here the following result used in the proof
of Theorem \ref{THM:Hardy,N,Omega}.

\begin{proposition}
Let $\Pi_1$, $\Pi_2$ be two half-spaces in the Heisenberg group
$\Hei^n$ and let $d_{N,1}$, $d_{N,2}$
denote the correspoding distances to the boundary with respect to
the gauge quasi-norm $N$, cf. \eqref{ntor}. If
there exists $\xi\in\Hei^n$ such that
\[
d_{N,1}(\xi)=d_{N,2}(\xi) \qquad \mbox{ and } \qquad 
\nabla d_{N,1}(\xi)=\nabla d_{N,2}(\xi),
\]
then $\Pi_1=\Pi_2$.
\end{proposition}

\proof
By group action there exist $\xi_1,\xi_2\in\Hei^n$ such that
\[
d_{N,k}(\xi) = {\rm dist}_N( \xi_k \, \xi , \partial \Pi_0) 
=:d_N(\xi_k \, \xi) , \quad k=1,2, \; \; \xi\in\Pi_k,
\]
where $\Pi_0$
and $d_N$ is as in Lemma \ref{lem:d_Nformula}.

Given $\xi\in\Pi_1\cap \Pi_2$ we set
\[
\xi_k  \xi =: \eta_k  =:(\alpha_k,\beta_k,\tau_k) \; , \qquad k=1,2.
\]
By Lemma \ref{lem:d_Nformula} $d_N(\eta_k)$ depends only on
$\rho_k:=\sqrt{|\alpha_k|^2+|\beta_k|^2}$ and $\tau_k$. Let $s_k$, $k=1,2$, be defined by
\be 
\la{aa5}
s_k^3 +2s_k =\frac{\tau_k}{\rho_k^2}.
\ee
If $d_{N,1}(\xi)=d_{N,2}(\xi)$ then $d_{N}(\eta_1)=d_{N}(\eta_2)$ and recalling also Remark \ref{kos} we obtain
\be 
\la{aa1}
\rho_1^4 \, s_1^4 \, (s_1^2+1) = \rho_2^4 \, s_2^4 \, (s_2^2+1).
\ee
Writing $\xi=(x,y,t)$ and $\xi_k=(x_k,y_k,t_k)$ we have
\[
\eta_k =\big( x_k + x , y_k+y ,t_k+t +2(x_k\cdot y -x\cdot y_k) \big).
\]
Assume now that we  additionally
have $\nabla d_{N,1}(\xi)=\nabla d_{N,2}(\xi)$. 
There holds
\[
\parder{d_{N,k}}{t}(\xi) =  \parder{d_N}{\tau_k}(\eta)\,.
\]
Hence, using again Remark \ref{kos}, 
\be 
\la{aa2}
\rho_1^2 \, s_1^3  = \rho_2^2 \, s_2^3 .
\ee
Combining \eqref{aa1} and \eqref{aa2} and recalling that $s_1$ and $s_2$
are positive we obtain that
\be
s_1=s_2=:s \; , \qquad \qquad  \rho_1=\rho_2=:\rho.
\la{aa3}
\ee
Now, since $d_N(\eta_k)$ depends only on $\rho_k$ and $\tau_k$, we have by Remark \ref{kos} that 
\begin{align*}
\parder{d_{N,k}}{x_i}(\xi)& = \parder{d_N}{\rho_k}
\parder{\rho_k}{x_i} + \parder{d_N}{\tau_k}
\parder{tau_k}{x_i} \\
& =  \frac{1}{\rho_k} \parder{d_N}{\rho_k}(x_i +x_{k,i}) -
2 \parder{d_N}{\tau_k} y_{k,i} \\
&=  -\frac{ \rho_k^2 s_k^4}{d_{N,k}^3}(x_i +x_{k,i}) -
\frac{ \rho_k^2 s_k^3}{2d_{N,k}^3} y_{k,i},
\end{align*}
and similarly
\[
\parder{d_{N,k}}{y_i}(\xi) =   -\frac{ \rho_k^2 s_k^4}{d_{N,k}^3}(y_i +y_{k,i}) +\frac{ \rho_k^2 s_k^3}{2d_{N,k}^3} x_{k,i}.
\]
Recalling also \eqref{aa3}, it follows that at a point $\xi$ as in the statement we have
\begin{align*}
& -s(x_i +x_{1,i})  -\frac{ y_{1,i}}{2} = 
-s(x_i +x_{2,i})  -\frac{ y_{2,i}}{2} \\
& -s(y_i +y_{1,i}) +\frac{ x_{1,i}}{2} = 
-s(y_i +y_{2,i})  +\frac{ x_{2,i}}{2}
\end{align*}
and therefore $x_1=x_2$ and $y_1=y_2$.
Finally, from \eqref{aa5} we also obtain $\tau_1=\tau_2$ and therefore $t_1=t_2$. Hence $\xi_1=\xi_2$ as required.
\endproof

\medskip

\noindent
{\bf Acknowledgment.} We would like to thank the two anonymous
referees for the careful reading of the paper and for
their useful comments.


\end{document}